\documentclass[10pt,a4paper]{article}

\usepackage[utf8]{inputenc}
\usepackage[english]{babel}
\usepackage[T1]{fontenc}
\usepackage{amsmath}
\usepackage{amsfonts}
\usepackage{amssymb}
\usepackage{amsthm}
\usepackage{stmaryrd}
\usepackage{tikz-cd}
\usepackage{enumitem}
\usepackage[disable]{todonotes}
\usepackage{graphicx}

\DeclareMathOperator{\Ker}{Ker}

\DeclareMathOperator{\disc}{disc}
\DeclareMathOperator{\Inv}{Inv}

\DeclareMathOperator{\Id}{Id}
\DeclareMathOperator{\Pf}{Pf}
\DeclareMathOperator{\Quad}{Quad}

\newcommand{\Isom}{\stackrel{\sim}{\longrightarrow}}

\newcommand{\Z}{\mathbb{Z}}
\newcommand{\N}{\mathbb{N}}
\newcommand{\Q}{\mathbb{Q}}
\newcommand{\R}{\mathbb{R}}

\newcommand{\pfis}[1]{\langle\!\langle #1\rangle\!\rangle}
\newcommand{\gpfis}[1]{\langle | #1 |\rangle}
\newcommand{\To}{\longrightarrow}

\newcommand{\fdiag}[1]{\langle #1\rangle}
\newcommand{\ens}[2]{\{ #1\,|\, #2\}}

\newcommand{\tld}{\widetilde}
\newcommand{\eps}{\varepsilon}

\newcommand{\bitem}{\item[$\bullet$]}
\newcommand{\pgq}{\geqslant}
\newcommand{\ppq}{\leqslant}

\newcommand{\minus}{\scalebox{0.75}[1.0]{$-$}}

\renewcommand{\phi}{\varphi}
\renewcommand{\hat}{\widehat}
\renewcommand{\bar}{\overline}

\newcommand{\foncdef}[5]{\begin{array}{rrcl}
#1 : & #2 & \To & #3 \\
 & #4 & \longmapsto & #5
\end{array}}

\newcommand{\anonfoncdef}[4]{\begin{array}{rcl}
#1 & \To & #2 \\
#3 & \longmapsto & #4
\end{array}}

\newcommand{\isomdef}[5]{\begin{array}{rrcl}
#1 : & #2 & \Isom & #3 \\
 & #4 & \longmapsto & #5
\end{array}}

\newtheorem{thm}{Theorem}[section]
\newtheorem{prop}[thm]{Proposition}
\newtheorem{coro}[thm]{Corollary}
\newtheorem{lem}[thm]{Lemma}
\newtheorem{defi}[thm]{Definition}

\newtheorem{propdef}[thm]{Proposition-definition}

\theoremstyle{definition}
\newtheorem{rem}[thm]{Remark}
\newtheorem{ex}[thm]{Example}

\author{Nicolas Garrel}
\title{Witt and Cohomological Invariants of Witt Classes}

\begin{document}

\maketitle

\begin{abstract}
  We describe all Witt invariants and mod 2 cohomological invariants
  of the functor $I^n$ as combinations of fundamental invariants ; this
  is related to the study of operations on mod 2 Milnor K-theory.
  We also study behaviour of these invariants with respect to
  products, restrictions, similitudes and ramification.


\end{abstract}

\section*{Introduction}

Building on classical constructions such as the discriminant
and the Hasse-Witt invariant, cohomological invariants have become
a standard tool in the study of quadratic forms. Cohomological
invariants of quadratic forms are also related to cohomological
invariants of algebraic groups, for split groups of orthogonal type.

In \cite{GMS}, Serre introduces cohomological invariants over a field,
and completely describes (away from characteristic 2)
the invariants of $\Quad_n$ (non-degenerated $n$-dimensional quadratic forms)
and $\Quad_{n,\delta}$ (those with prescribed determinant $\delta$),
and in particular this settles the case of invariants of split orthogonal
and special orthogonal groups. In contrast, the case of split
spin groups, corresponding to invariants of $\Quad_n\cap I^3$ (meaning
that the Witt classes of the forms must be in $I^3$), is very much open,
and has only been treated for small $n$ (see for instance \cite{Gar})
or for invariants of small degree (the case of degree 3 has been
essentially solved by Merkurjev in \cite{Mer}); one problem being that we do
not have any satisfying parametrization of $\Quad_n\cap I^3$.

On the other hand, if we move from isometry classes to Witt classes,
following the resolution of Milnor's conjecture by Voevodsky, we
have at hand good descriptions of $I^n$ (see for instance \cite{EKM}),
and at least one important cohomological invariant of $I^n$,
$e_n : I^n(K)\to H^n(K,\mu_2)$. The goal of this article is to
describe all mod 2 cohomological invariants of $I^n$, and study
some of their basic properties.

Our starting point is a construction of Rost (\cite{R99}), who defines
a certain natural operation $P_n: I^n(K)\to I^{2n}(K)$ which behaves like
a divided square in the sense that $P_n(\sum \phi_i) = \sum_{i<j}\phi_i\cdot \phi_j$
if $\phi_i$ are $n$-fold Pfister forms. After composing with
$e_{2n}$ this gives a cohomological invariant of $I^n$ of degree $2n$.
We generalize this to operations $\bar{\pi}_n^d:I^n\to I^{dn}$ for all $d\in \N$
and thus cohomological invariants of degree $dn$. Since
our constructions involve both Witt invariants and cohomological
invariants, in order to avoid repeating very similar proofs in both
settings, we choose to adopt a unified point of view and treat
both cases simultaneously, using $A$ to denote either the Witt ring
or mod 2 cohomology.

We define two sets of generators for invariants, $f_n^d$ (mentioned above,
see \ref{prop_f}) and $g_n^d$ (definition \ref{def_g}),
each being useful depending on the situation. The invariants $g_n^d$ have the important
property that only a finite number of them are non-zero on a fixed
form (proposition \ref{prop_g_borne}), which allows to take infinite combinations,
and we show that any invariant of $I^n$ is equal to such a combination
(theorem \ref{thm_g}). They are also better behaved with respect
to similitudes (proposition \ref{prop_simil}). On the other hand, the
$f_n^d$ are preferable for handling products (proposition \ref{prop_prod} and
corollary \ref{cor_prod_cohom}) and restriction to $I^{n+1}$
(corollaries \ref{cor_restr_w} and \ref{cor_restr_cohom}). We also
study behaviour with respect to residues from discrete valuations
(proposition \ref{prop_ram}), and establish links with Serre's
description of invariants of isometry classes (proposition
\ref{prop_fixed_dim}).

Our invariants may be related to other various constructions on
Milnor K-theory and Galois cohomology, notably by Vial in \cite{Via}.
The invariants defined here may be seen as lifting of Vial's to the level of $I^n$.
See section \ref{sec_operations} for more details.

Finally, we adapt an idea of Rost (\cite{R99}, see also Garibaldi
in \cite{Gar}) to study invariants of Witt classes in $I^n$ that are divisible
by a $r$-fold Pfister form, giving a complete description for $r=1$
(theorem \ref{thm_delta}).

\section*{Notations and some preliminaries}

In all that follows, $k$ is a fixed field of characteristic
different from 2, and $K$ denotes any field extension of $k$.
The set of natural integers is denoted by $\N$, and the positive
integers by $\N^*$; if $x\in \R$, $\lfloor x\rfloor\in \Z$ will denote
its floor, and $\lceil x\rceil$ its ceiling. We extend the binomial
coefficient $\binom{a}{b}$ for arbitrary $a,b\in \Z$ in the only way
that still satisfies Pascal's triangle.

For all facts on quadratic forms, the reader is referred to \cite{EKM}.
All the quadratic forms we consider are assumed to be non-degenerated.
The Grothendieck-Witt ring $GW(K)$ has a fundamental ideal $\hat{I}(K)$,
defined as the kernel of the dimension map $GW(K)\to \Z$.
We denote by $[q]\in W(K)$ the Witt class of an element $q\in GW(K)$,
and this ring morphism $GW(K)\to W(K)$ induces an isomorphism between
$\hat{I}(K)$ and the fundamental ideal $I(K)\subset W(K)$.
If $x\in I(K)$, we write $\hat{x}\in \hat{I}(K)$ for its
(unique) antecedent. If $n\in \N$ and $q\in W(K)$, then
$nq = q + \cdots + q$ is not to be confused with $\fdiag{n}q$
which is pointwise multiplication by the scalar $n\in K^*$.

If $a\in K^*$, then we write $\pfis{a} = \fdiag{1,-a}\in I(K)$,
and if $a_1,\dots,a_n\in K^*$ then
$\pfis{a_1,\dots,a_n}=\pfis{a_1}\cdots \pfis{a_n}\in I^n(K)$.
Those elements are (the Witt classes of) the $n$-fold Pfister
forms, and we use $\Pf_n(K)\subset I^n(K)$ for the set of
such elements. We also write $\gpfis{a_1,\dots,a_n}$
for the antecedent of $\pfis{a_1,\dots,a_n}$ in $\hat{I}^n(K)$;
we call such elements $n$-fold Grothendieck-Pfister elements,
and we write $\hat{\Pf}_n(K)\subset \hat{I}^n(K)$ for their set.
For instance $\gpfis{a} = \fdiag{1}-\fdiag{a}$, so $\gpfis{1}=0$.
Notice that if $q\in W(K)$, then $2q = \pfis{-1}q$,
and in particular if $-1$ is a square in $K$ then $2q=0$ in $W(K)$.
Also, if $\phi\in \Pf_n(K)$, then $\phi^2=2^n\phi$, since
$\pfis{a,a} = \pfis{-1,a} = 2 \pfis{a}$. This relation is also
true if $\phi\in \hat{\Pf}_n(K)$.

By a filtered group $A$ we mean that there are subgroups $A^{\pgq d}$
for all $d\in \Z$, such that $A^{\pgq d+1}\subset A^{\pgq d}$. We say
the filtration is \emph{positive} if $A^{\pgq d}=A$ for all $d\ppq 0$,
and that it is \emph{separated} if $\bigcap_d A^{\pgq d}=0$.
If $A$ is a ring, it is a filtered ring if
$A^{\pgq d}\cdot A^{\pgq d'}\subset A^{\pgq d+d'}$, and $M$ is a filtered
$A$-module if it is a filtered group such that
$A^{\pgq d}\cdot M^{\pgq d'}\subset M^{\pgq d+d'}$.
For any $n\in \Z$, we denote by $M[n]$ the filtered module such that
$(M[n])^{\pgq d}=M^{\pgq d+n}$. A morphism of filtered modules $f:M\to N$
is a module morphism such that $f(M^{\pgq d})\subset N^{\pgq d}$.

Let $\mathrm{Fields}_{/k}$ be the category of field extensions of $k$.
If we are given functors $T:\mathrm{Fields}_{/k} \to \mathrm{Sets} $ 
and $A: \mathrm{Fields}_{/k}\to \mathrm{Ab}$ (the category
of abelian groups), then an invariant of $T$ with values in $A$
(over $k$) is a natural transformation from $T$ to $A$. The set of
such invariants is naturally an abelian group, denoted $\Inv(T,A)$.
If $T$ takes values in \emph{pointed} sets, then we can define
\emph{normalized} invariants as the ones that send the distinguished
element to $0$. This subgroup is denoted $\Inv_0(T,A)$,
and we have $\Inv(T,A) = A(k)\oplus \Inv_0(T,A)$.

Since we want to unify proofs for Witt and cohomological invariants,
we will use $A(K)$ for either $W(K)$ or $H^*(K,\mu_2)$ (we write $A=W$ or
$A=H$ if we want to distinguish cases). For $d\in \N$, we set $A^{\pgq d}(K) = I^d(K)$
if $A=W$, and $A^{\pgq d}(K) =\bigoplus_{i\pgq d}H^i(K,\mu_2)$ if $A=H$.
Then $A(K)$ is a filtered $A(k)$-algebra, and the filtration is separated
and positive. Note that according to the resolution of Minor's conjecture
by Voevodsky et al., the graded ring associated to $A(K)$ is in both cases
the mod 2 cohomology ring $H^*(K,\mu_2)$.

For any $n\in \N^*$, we write $M(n)=\Inv(I^n,A)$, and
$M^{\pgq d}(n)= \Inv(I^n,A^{\pgq d})$ for all $d\in \N$. Similarly,
the subgroups of normalized invariants are denoted $M_0(n)$ and
$M_0^{\pgq d}(n)$. Then $M(n)$ is a filtered $A(k)$-algebra, and $M_0(n)$
is a submodule.

We list here the formal properties of $A$ on which the article relies.
We have a group morphism $f_n:I^n(K)\to A^{\pgq n}(K)$ (either the identity
if $A=W$, or the morphism $e_n$ given by the Milnor conjecture if $A=H$)
and we write $\{a_1,\dots,a_n\} = f_n(\pfis{a_1,\dots,a_n})$
(so it is either a Pfister form or a Galois symbol depending on $A$).
Note that
\begin{equation}\label{eq_prod_fn}
  f_n(x)\cdot f_m(y) = f_{n+m}(xy).
\end{equation}

We set $\delta = \delta(A) = 1$ if $A=W$, and $\delta = 0$
if $A=H$. Then we have
\begin{equation}\label{eq_delta_sum}
  \forall a,b\in K^*, \{ab\} = \{a\} + \{b\} -\delta \{a,b\}
\end{equation}
and
\begin{equation}\label{eq_delta_2}
  \delta \{-1\} = 2 \in A(K).
\end{equation}

We will also freely use the following lemmas:

\begin{lem}\label{lem_a_filtr}
  If $x\in A(K)$ is such that for any extension $L/K$ and any
  $\phi\in \Pf_n(L)$ we have $f_n(\phi)\cdot x\in A^{\pgq d+n}(L)$,
  then $x\in A^{\pgq d}(K)$. In particular, for any $n\in \N^*$,
  if $f_n(\phi)\cdot x=0$ for all $\phi\in \Pf_n(L)$, then $x=0$.
\end{lem}

\begin{lem}\label{lem_a_inv}
  $\Inv(\Pf_n,A) = A(K)\oplus A(K)\cdot f_n$ where we consider
  invariants defined over $K$.
\end{lem}

The first lemma can be proved by specialisation, taking $\phi$ to be
a generic Pfister form; the second corresponds to two theorems of Serre
(\cite[thm 18.1, ex 27.17]{GMS}).

\section{Some pre-$\lambda$-ring structures}

We refer to \cite{Yau} for the basic theory of $\lambda$-rings.
If $R$ is a commutative ring, a pre-$\lambda$-ring structure on
$R$ is the data of applications $\lambda^d:R\to R$ for all $d\in \N$
such that for all $x,y\in R$:
\begin{enumerate}[label=(\roman*)]
  \item $\lambda^0(x) = 1$;
  \item $\lambda^1(x) = x$;
  \item $\forall d\in\N,\, \lambda^d(x+y) = \sum_{k=0}^d \lambda^k(x)\lambda^{d-k}(y)$.
\end{enumerate}

\begin{ex}
  The example we are interested in is $R=GW(K)$. The $\lambda^d$ are the
  exterior powers of bilinear forms, as defined in \cite{Bou}, and it is
  shown in \cite{McGar} that they define a $\lambda$-ring structure on
  $GW(K)$ (which is a pre-$\lambda$-ring structure with additional
  conditions).
\end{ex}

We define $\Lambda(R)=1+tR[[t]]$, the subset of formal power series
with coefficients in $R$ that have a constant coefficient equal to $1$.
It is a group for the multiplication of formal series.
If we set $\lambda_t(x) = \sum_{d\in \N} \lambda^d(x)t^d \in R[[t]]$,
we see that a pre-$\lambda$-ring structure on $R$ is equivalent to the
data of a group morphism $\lambda_t: (R,+)\to (\Lambda(R),\cdot)$ such
that for all $x\in R$ the degree 1 coefficient of $\lambda_t(x)$ is $x$.
We will switch freely between those two descriptions.

\begin{ex}
  For the canonical $\lambda$-ring structure on $GW(K)$, we have
  $\lambda_t(\fdiag{a}) = 1+\fdiag{a}t$ for all $a\in K^*$. 
\end{ex}

Recall that for any formal series $f,g\in R[[t]]$ such that the constant coefficient of
$f$ is zero, we can define the composition $g\circ f \in R[[t]]$. If furthermore the
degree 1 coefficient of $f$ is invertible in $R$, then $f$ as an inverse for
the composition, which we denote $f^{\circ -1}$.

\begin{lem}\label{lem_lambda_compo}
  Let $R$ be a commutative ring. If $\lambda_t: R\to \Lambda(R)$ defines
  a pre-$\lambda$-ring structure on $R$, then for any $f\in t+t^2R[[t]]$,
  the map
  \[ \foncdef{\lambda_{f(t)}}{R}{\Lambda(R)}{x}{\lambda_t(x)\circ f = \sum_{d\in \N} \lambda^d(x)f(t)^d} \]
  also defines a pre-$\lambda$-ring structure.
\end{lem}

\begin{proof}
  We have for any $x,y\in R$:
  \begin{align*}
    \lambda_t(x+y)\circ f &= (\lambda_t(x)\lambda_t(y))\circ f \\
    &= (\lambda_t(x)\circ f)\cdot (\lambda_t(y)\circ f).
  \end{align*}
  Furthermore, since the degree 1 term of $f(t)$ is $t$, the degree 1 coefficient
  of $\lambda_t(x)\circ f$ is the same as that of $\lambda_t(x)$, which is $x$.
\end{proof}

We want to define for each $n\in \N^*$ a pre-$\lambda$-ring structure on $GW(K)$
that vanishes for $d\pgq 2$ on $n$-fold Grothendieck-Pfister elements.
Our starting point is the following fundamental observation:

\begin{lem}\label{lem_lambda_gpfis}
  Let $a\in K^*$. For any $d\pgq 1$, we have $\lambda^d(\gpfis{a}) = \gpfis{a}$.
  Therefore,
  \[ \lambda_t(\gpfis{a}) = 1 + \gpfis{a}x(t) \]
  where $x(t) = \sum_{d\pgq 1}t^d = \frac{t}{1-t}$.
\end{lem}

\begin{proof}
  We have
  \begin{align*}
    \lambda_t(1-\fdiag{a}) &= \frac{\lambda_t(1)}{\lambda_t(\fdiag{a})} \\
                           &= \frac{1+t}{1+\fdiag{a}t} \\
                           &= 1 + \sum_{d\pgq 1} (1-\fdiag{a})t^d. 
  \end{align*}
  using $\fdiag{a}^2=1$.
\end{proof}

We then define some formal series, for any $n\in \N^*$: $x_n(t)\in \Z[[t]]$
is defined recursively by
\[ x_1(t) = \frac{t}{1-t},\quad x_{n+1} = x_n + 2^{n-1}x_n, \]
and $h_n(t)\in \Q[[t]]$ by
\[ h_n = x_n^{\circ -1}. \]

\begin{lem}\label{lem_a_b_h}
  For any $n\in \N^*$, we have $h_n(t)\in \Z[[t]]$. Furthermore,
  if $a_n$ and $b_n$ are respectively the even part and odd part
  of $x_n$, then:
  \[ \left\{ \begin{array}{l}
    a_{n+1} = 2^n b_n^2 = 2a_n + 2^n a_n^2 \\
    b_{n+1} = b_n + 2^n a_n b_n.
    \end{array} \right. \]
\end{lem}

\begin{proof}
  Note first that $h_1(t) = \frac{t}{1+t}\in \Z[[t]]$. Let
  $p_n(t) = t + 2^{n-1}t^2\in \Z[[t]]$; then by definition $x_{n+1}=p_n\circ x_n$,
  so $h_{n+1} = h_n\circ p_n^{\circ -1}$. Now a simple computation yields
  $p_n^{\circ -1} = tC(-2^{n-1}t)$ where $C(t) = \frac{1- \sqrt{1-4t}}{2t}$
  is the generating function of the Catalan numbers (this is essentially equivalent
  to the well-known functional equation for $C(t)$); in particular, $p_n^{\circ -1}$
  has integer coefficients, so $h_n(t)\in \Z[[t]]$.

  Separating even and odd parts, the recursive definition of $x_n$ yields
  \[ \left\{ \begin{array}{l}
    a_{n+1} = a_n + 2^{n-1}a_n^2 + 2^{n-1}b_n^2 \\
    b_{n+1} = b_n + 2^n a_n b_n.
  \end{array} \right. \]
  So we need to show that for any $n\in \N^*$, $a_n+2^{n-1}a_n^2=2^{n-1}b_n^2$.
  If $n=1$, this is a direct computation, using that $a_1(t) = \frac{t^2}{1-t^2}$
  and $b_1(t) = \frac{t}{1-t^2}$.

  Now suppose the formula holds until $n\in \N^*$. Then
  \[ a_{n+1} + 2^n a_{n+1}^2 = 2^n b_n^2 + 2^n (2^{n}b_n^2)^2 = 2^n b_n^2(1 + 2^{2n}b_n^2) \]
  and
  \begin{align*}
    2^n b_{n+1}^2 &= 2^n b_n^2(1+2^n a_n)^2 \\
    &= 2^n b_n^2(1 + 2^{n+1}a_n + 2^{2n}a_n^2) \\
    &= 2^n b_n^2(1 + 2^{2n}b_n^2),
  \end{align*}
  which shows the expected formula.
\end{proof}

We can now use those formal series to define our pre-$\lambda$-ring
structures:

\begin{thm}\label{thm_pi}
  For any $n\in \N^*$, the map $(\pi_n)_t = \lambda_{h_n(t)}$ defines
  a pre-$\lambda$-ring structure on $GW(K)$ such that $\pi_n^d(\phi)=0$
  for any $\phi\in \hat{\Pf}_n(K)$ and any $d\pgq 2$.
\end{thm}

\begin{proof}
  According to lemma \ref{lem_lambda_compo}, $(\pi_n)_t$ does define
  a pre-$\lambda$-ring structure on $GW(K)$. We show the statement about
  Grothendieck-Pfister elements by induction on $n$. For $n=1$, the statement
  is equivalent to lemma \ref{lem_lambda_gpfis} since for any $\phi\in \hat{\Pf}_1(K)$,
  $\lambda_t(\phi) = 1 + \phi x_1(t)$ and $h_1 = x_1^{\circ -1}$.

  Suppose the statement holds until $n\in \N^*$. Let $\phi\in \hat{\Pf}_{n+1}(K)$,
  and write $\phi=\gpfis{a}\psi$ with $a\in K^*$ and $\psi\in \hat{\Pf}_n(K)$.
  We then need to show $\lambda_{h_{n+1}(t)}(\phi)=1+\phi t$, which is equivalent to
  \[ \lambda_t(\gpfis{a}\psi) = 1 + \gpfis{a}\psi x_{n+1}(t). \]

  Note that for any $x\in \hat{I}(K)$, we have $-\fdiag{a}x = \fdiag{-a}x$,
  which implies that $\lambda^d(-\fdiag{a}x) = (-1)^d\fdiag{a^d}\lambda^d(x)$ for
  any $d\in \N$, and thus $\lambda_t(-\fdiag{a}x)=\lambda_{-\fdiag{a}t}(x)$.
  Therefore we have in $GW[[t]]$:
  \begin{align*}
     \lambda_t(\psi - \fdiag{a}\psi) &= \lambda_t(\psi)\lambda_{-\fdiag{a}t}(\psi) \\
           &= (1 + \psi x_n(t))(1 + \psi x_n(-\fdiag{a}t)) \\
           &= 1 + \psi\left( x_n(t) + x_n(-\fdiag{a}t) + 2^nx_n(t)x_n(-\fdiag{a}t)\right).
  \end{align*}
  Thus we can conclude if we show that
  \[  x_n(t) + x_n(-\fdiag{a}t) + 2^nx_n(t)x_n(-\fdiag{a}t) = (1-\fdiag{a})x_{n+1}(t). \]
  If we decompose in even and odd parts, this amounts to
  \[ \left\{ \begin{array}{l}
    a_n(t) + a_n(t) + 2^n(a_n(t)^2 -\fdiag{a}b_n(t)^2) = (1-\fdiag{a})a_{n+1}(t) \\
    b_n(t) -\fdiag{a}b_n(t) + 2^n(b_n(t)a_n(t) - \fdiag{a}a_n(t)b_n(t)) = (1-\fdiag{a})b_{n+1}(t),
  \end{array} \right. \]
  which are consequences of lemma \ref{lem_a_b_h}.
\end{proof}

\begin{rem}
  Those are \emph{not} $\lambda$-ring structures; for instance, $\pi_n^d(1)\neq 0$
  for $d\pgq 2$.
\end{rem}

\begin{coro}\label{cor_pi_sum}
  Let $n\in \N^*$, and $\phi_1,\dots,\phi_r\in \hat{\Pf}_n(K)$. Then:
  \[ \pi_n^d\left( \sum_{i=1}^r \phi_i \right) = \sum_{1\ppq i_1<\dots <i_d\ppq r} \phi_{i_1}\dots \phi_{i_d}. \]
  In particular, $\pi_n^d(\hat{I}^n(K))\subset \hat{I}^{nd}(K)$, and
  $\pi_n^d$ is zero on forms that are sums of $d-1$ (or less)
  $n$-fold Grothendieck-Pfister elements.
\end{coro}

\begin{proof}
  The formula is proved by an easy induction, exactly similar to the proof
  of the formula for exterior powers of diagonal quadratic forms (or
  more generally $\lambda$-powers of a sum of elements of dimension 1
  in any pre-$\lambda$-ring). If $x\in \hat{I}^n(K)$, then $x=x_1-x_2$
  where the $x_i$ are sums of elements of $\hat{\Pf}_n(K)$, and
  $(\pi_n)_t(x) = (\pi_n)_t(x_1)\cdot ((\pi_n)_t(x_2))^{-1}$. Now it is easy
  to see that since the degree $d$ coefficient of $(\pi_n)_t(x_i)$ is
  in $\hat{I}^{nd}(K)$, then the same is true for $(\pi_n)_t(x)$.
\end{proof}

Note that the formula in corollary \ref{cor_pi_sum} is not enough to completely
describe $\pi_n^d$ on $\hat{I}^n(K)$, even if we could show directly that
it is well-defined (which is possible using the presentation of $I^n(K)$
given in \cite[42.4]{EKM}), since not every element of $I^n(K)$ is a sum
of Pfister forms.

The idea of similar ``divided power'' operations on related structures
such as Milnor K-theory of Galois cohomology has been around for some
time (see section \ref{sec_operations} for more details).

\section{The fundamental invariants}

We will now use these various pre-$\lambda$-ring structures on $GW(K)$
to define some invariants of $I^n$.

\begin{defi}
  Let $n\in \N^*$ and $d\in \N$. Then we define
  \[ f_n^d: I^n(K) \Isom \hat{I}^n(K) \xrightarrow{\pi_n^d} \hat{I}^{nd}(K) \Isom I^{nd}(K)
  \xrightarrow{f_{nd}} A^{\pgq nd}(K). \]
  
  If $A=W$, then we will sometimes write $f_n^d=\bar{\pi}_n^d$.
  
  If $A=H$, then we will sometimes write $f_n^d = u_{nd}^{(n)}$.
\end{defi}

This is well-defined according to corollary \ref{cor_pi_sum}.
The notation $u_{nd}^{(n)}$ may seem dissonant with the rest, but we chose to stick with
the tradition to write the degree of cohomological invariants in the
index, and the exponent serves to distinguish between, for instance,
$u_6^{(2)}: I^2(K)\to H^6(K,\mu_2)$ and $u_6^{(3)}: I^3(K)\to H^6(K,\mu_2)$,
which are completely different ($u_6^{(3)}$ is \emph{not} the restriction
of $u_6^{(2)}$ to $I^3$).

\begin{prop}\label{prop_f}
  Let $n\in \N^*$. Then for any $d\in \N$, we have $f_n^d\in M^{\pgq nd}(n)$,
  and $(f_n^d)_{d\in \N}$ is the only family of elements of
  $M(n)$ such that:
  \begin{enumerate}[label=(\roman*)]
  \item $f_n^0 = 1$ and $f_n^1 = f_n$;
  \item for all $q,q'\in I^n(K)$:
    \[ f_n^d(q+q') = \sum_{k=0}^d f_n^k(q)\cdot f_n^{d-k}(q') ; \]
  \item for all $\phi\in \Pf_n(K)$ and $d\pgq 2$, $f_n^d(\phi)=0$.
  \end{enumerate}
  Furthermore, for any $\phi\in \Pf_n(K)$ and any $d\in \N^*$:
  \begin{equation}\label{eq_fn_moins_phi}
    f_n^d(-\phi) = (-1)^d\{-1\}^{n(d-1)}f_n(\phi).
  \end{equation}
\end{prop}

\begin{proof}
  The fact that $f_n^d$ is an invariant is clear by construction:
  the definition of $\pi_n^d$ is made in terms of the exterior powers,
  which are of course compatible with field extensions, and the expression
  of the $\pi_n^d$ in terms of the $\lambda^d$ is given by a universal
  $h_n\in \Z[[t]]$.

  The three properties are direct consequences of theorem \ref{thm_pi},
  after applying $f_{nd}$ to the corresponding formulas for $\pi_n^d$
  (and using formula (\ref{eq_prod_fn})).

  The last formula on opposites of Pfister forms can be easily proved
  by induction using
    \[ 0 = f_n^d(\phi-\phi) = f_n^d(-\phi) + f_n^{d-1}(-\phi)f_n(\phi). \]
  
  Uniqueness follows from property $(ii)$ and the fact that Pfister forms additively
  generate $I^n(K)$, since the values of $f_n^d$ are fixed on $\pm \phi$
  for any $\phi\in \Pf_n(K)$. 
\end{proof}

As an immediate consequence of either corollary \ref{cor_pi_sum} or proposition \ref{prop_f}:

\begin{coro}\label{cor_sum_f}
  Let $n\in \N^*$, and $\phi_1,\dots,\phi_r\in \Pf_n(K)$. Then:
  \[ f_n^d\left( \sum_{i=1}^r \phi_i \right) =
          \sum_{1\ppq i_1<\dots <i_d\ppq r} f_n(\phi_{i_1})\dots f_n(\phi_{i_d}). \]
  In particular, $f_n^d$ is zero on forms that are sums of $d-1$ or less
  $n$-fold Pfister forms.
\end{coro}

\section{The shifting operator}

Since $I^n(K)$ is additively generated by the $n$-fold Pfister forms, it is natural
to study how the invariants behave under adding or subtracting a Pfister form.

\begin{propdef}
  Let $n\in \N^*$ and $\eps=\pm1$. There is a unique morphism of filtered $A(k)$-modules
  $\Phi_n^\eps: M(n)\to M(n)[-n]$ such that
  \begin{equation*}
    \alpha(q + \eps\phi) = \alpha(q) + \eps f_n(\phi)\cdot \Phi_n^\eps(\alpha)(q)
  \end{equation*}
  for all $\alpha\in M(n)$, $q\in I^n(K)$ and $\phi\in \Pf_n(K)$. 
\end{propdef}

\begin{proof}
  Let $\alpha\in M(n)$ and $q\in I^n(K)$. For any extension $L/K$
  and any $\phi\in \Pf_n(L)$, we set
  \[ \beta_q(\phi) = \alpha(q + \eps\phi). \]

  Then $\beta_q\in \Inv(\Pf_n,A)$, defined over $K$. According to
  lemma \ref{lem_a_inv}, there are uniquely determined $x_q,y_q\in A(K)$
  such that $\beta_q = x_q + y_q\cdot f_n$.

  Taking $\phi=0$ we see that $x_q = \alpha(q)$, and we then set
  $\Phi_n^\eps(\alpha)(q) = \eps y_q$, which gives the expected formula, as well as
  the uniqueness of $\Phi_n^\eps$.

  By definition, $\Phi_n^\eps$ is clearly a $A(k)$-module morphism,
  and it is of degree $-n$ because if $\alpha\in M^{\pgq d}(n)$, then for any
  $q\in I^n(K)$, $f_n(\phi)\cdot \alpha^\eps(q)\in A^{\pgq d}(L)$ for all $\phi\in Pf_n(L)$
  and any extension $L/K$, thus $\alpha^\eps(q)\in A^{\pgq d-n}(K)$ by lemma
  \ref{lem_a_filtr}.
\end{proof}

We will often write $\Phi^+=\Phi_n^{+1}$ and $\Phi^-=\Phi_n^{-1}$, as
there is in practice no confusion to what $n$ is in the context. We also
write $\alpha^+ = \Phi^+(\alpha)$ and $\alpha^- = \Phi^-(\alpha)$ for
any $\alpha\in M(n)$. These two operators have natural links between
each other:

\begin{prop}\label{prop_phi_pm}
  Let $n\in \N^*$. The operators $\Phi_n^+$ et $\Phi_n^-$ commute, and furthermore
  for any $\alpha\in M(n)$ we have: 
  \[ \alpha^+ - \alpha^- = \{-1\}^n \alpha^{+-} = \{-1\}^n \alpha^{-+}. \]
\end{prop}

\begin{proof}
  Let $q\in I^n(K)$ and $\phi,\psi\in \Pf_n(L)$.
  We have
  \begin{align*}
    \alpha(q+\phi-\psi) &= \alpha(q+\phi) - f_n(\psi) \alpha^-(q+\phi) \\
                        &= \alpha(q) + f_n(\phi)\alpha^+(q) -f_n(\psi)\alpha^-(q)
                          -f_n(\phi)f_n(\psi)\alpha^{-+}(q)
  \end{align*}
  but also
  \begin{align*}
    \alpha(q+\phi-\psi) &= \alpha(q-\psi) + f_n(\phi) \alpha^+(q-\psi) \\
                        &= \alpha(q) - f_n(\psi)\alpha^-(q) +f_n(\phi)\alpha^+(q)
                          -f_n(\phi)f_n(\psi)\alpha^{+-}(q)
  \end{align*}
  thus $f_n(\phi)f_n(\psi)\alpha^{-+}(q) = f_n(\phi)f_n(\psi)\alpha^{+-}(q)$,
  and since this holds for any $\phi$, $\psi$ over any extension,
  by lemma \ref{lem_a_filtr} we find $\alpha^{+-}=\alpha^{-+}$.

  If we now take $\phi=\psi$, the above formula gives
  \[ f_n(\phi)\alpha^+(q) - f_n(\phi)\alpha^-(q) = f_n(\phi)f_n(\phi)\alpha^{+-}(q) \]
  which allows to conclude, using $f_n(\phi)f_n(\phi) = \{-1\}^nf_n(\phi)$
  and again \ref{lem_a_filtr}.
\end{proof}

In view of this proposition, we may write $\alpha^{r+,s-}\in M(n)$ for
any $\alpha\in M(n)$ and $r,s\in \N$, defined as applying $r$ times $\Phi^+$ to
$\alpha$, and $s$ times $\Phi^-$, in any order.

We also call $\Phi=\Phi^+$ the \emph{shifting operator}, as justified by the
following elementary result:

\begin{prop}\label{prop_phi_pi}
  Let $n\in \N^*$. For any $d\in \N$, $\Phi(f_n^{d+1})=f_n^d$ (and $\Phi(f_n^0)=0$).
\end{prop}

\begin{proof}
  We need to show $f_n^{d+1}(q+\phi)=f_n^{d+1}(q) + f_n(\phi)\cdot f_d^d(q)$,
  which is an immediate consequence of proposition \ref{prop_f}.
\end{proof}

The action of $\Phi^-$ on the $f_n^d$ is more complicated,
reflecting the fact that $f_n^d$ behaves very nicely with respect
to sums of Pfister forms, but quite poorly for difference of
those.

\begin{prop}\label{prop_fn_phi_moins}
  Let $n,d\in \N^*$. Then
  \[  (f_n^d)^- = \sum_{k=0}^{d-1} (-1)^{d-k-1} \{-1\}^{n(d-k-1)} f_n^k. \]
\end{prop}

\begin{proof}
  Let $q\in I^n(K)$ and $\phi\in \Pf_n(K)$. Then
  \begin{align*}
    f_n^d(q-\phi) &= \sum_{k=0}^df_n^k(q)f_n^{d-k}(-\phi) \\
                  &= f_n^d(q) + \sum_{k=0}^{d-1} (-1)^{d-k}\{-1\}^{n(d-k-1)}f_n(\phi)f_n^k(q)
  \end{align*}
  using formula (\ref{eq_fn_moins_phi}).
\end{proof}

Apart from its action on the $f_n^d$, the main property of
$\Phi_n^\eps$ is the following:

\begin{prop}\label{prop_ker_phi}
  Let $n\in \N^*$ and $\eps=\pm 1$. The morphism $\Phi_n^\eps$ induces for any $d\in \N$
  an exact sequence:
  \[ 0\To A(k)/A^{\pgq d+n}(k) \To M(n)/M^{\pgq d+n}(n)\xrightarrow{\Phi_n^\eps} M(n)/M^{\pgq d}(n). \]
  In particular, the kernel of $\Phi_n^\eps$ is the submodule of constant
  invariants in $M(n)$.
\end{prop}

\begin{proof}
  If $\alpha,\beta\in M(n)$ are congruent modulo $M^{\pgq d+n}(n)$, then since
  $\Phi^\eps(M^{\pgq d+n}(n))$ is included in $M^{\pgq d}(n)$, $\alpha^\eps$
  and $\beta^\eps$ are congruent modulo $M^{\pgq d}(n)$.

  Let $\alpha\in M(n)$ such that $\alpha^\eps\in M^{\pgq d}(n)$. Then for any $q\in I^n(K)$
  and any $\phi\in \Pf_n(K)$, we have $\alpha(q + \eps\phi)\equiv \alpha(q)$ modulo
  $A^{\pgq n+d}(K)$, and also by symmetry $\alpha(q - \eps\phi)\equiv \alpha(q)$.
  Since we can always write $q= q_1- q_2$ where the $q_i$ are sums of $n$-fold
  Pfister forms, then by simple induction on the lengths of the sums,
  $\alpha(q)\equiv \alpha(0)$ modulo $A^{\pgq n+d}(K)$
  (where $\alpha(0)$ is seen as a constant invariant).

  Taking a large enough $d$, and since the filtration on $A(K)$ is separated,
  we see that $\alpha^\eps=0$ implies $\alpha = \alpha(0)$.
\end{proof}

\begin{coro}\label{cor_phi_exact}
  Let $n\in \N^*$ and let $\eps=\pm 1$.
  If $M'(n)$ is the submodule of $M(n)$ generated by the $f_n^d$
  for $d\in \N$, then $\Phi_n^\eps$ induces an exact sequence
  of filtered $A(k)$-modules
  \[ 0 \to A(k) \To M'(n) \xrightarrow{\Phi_n^\eps} M'(n)[-n] \To 0. \]
\end{coro}

\begin{proof}
  The only thing left to check is surjectivity, but this is easily
  implied by propositions \ref{prop_phi_pi}
  for $\Phi^+$, and \ref{prop_fn_phi_moins} for $\Phi^-$.
\end{proof}

\begin{rem}\label{rem_shift}
  All this implies that $\Phi$ may be seen as some kind of differential operator:
  if we know $\alpha^+$ for some invariant $\alpha$, we may ``integrate''
  to find $\alpha$, with a certain integration constant.
  Precisely, if $\alpha^+ = \sum a_d f_n^d$, then
  $\alpha = \alpha(0)+\sum a_d f_n^{d+1}$ (and we will show
  in the next section that such a decomposition always holds).
  We will use extensively this method to compute some invariants $\alpha$
  by ``induction on shifting''.
\end{rem}

\section{Classification of invariants}

The main goal of this article, and this section, is to show that any
$\alpha\in M(n)$ can be expressed uniquely as a combination $\sum_d a_d f_n^d$.
The first step is:

\begin{prop}\label{prop_pi_gen}
  Let $n\in \N^*$ and $d\in \N$. The $A(k)/A^{\pgq d}(k)$-module $M(n)/M^{\pgq d}(n)$
  is generated by the $f_n^k$ with $nk<d$.
\end{prop}

\begin{proof}
  We use induction on $d$. For $d=0$, this is trivial
  since $M^{\pgq 0}(n)=M(n)$. Suppose the property holds up to
  $d-1$, and let $\alpha\in M(n)$; we set $\bar{\alpha}\in M(n)/M^{\pgq d}(n)$
  its residue class. By induction, $\Phi(\bar{\alpha}) = \sum a_k f_n^k$
  with $nk<d-n$, so if we set $\beta = \alpha -\sum a_k f_n^{k+1}$
  we get $\Phi(\bar{\beta})=0$. From there, $\beta$ is congruent modulo $M^{\pgq d}(n)$
  to a constant invariant $a_{-1}$, hence $\bar{\alpha} = \sum a_{k-1} f_n^k$
  with $nk<d$.
\end{proof}

The problem is that to express an invariant in terms of the $f_n^d$,
it is in general necessary to use an infinite
combination, as the following example illustrates.

\begin{ex}\label{ex_disc}
  Consider the case $A=W$. Let $\alpha(q) = \fdiag{\disc(q)}$;
  it is a Witt invariant of $I$. Then $\alpha^+ = -\alpha$; indeed:
  \[ \fdiag{\disc(q + \pfis{a})} = \fdiag{\disc(q)a} = \fdiag{\disc(q)} -\pfis{a}\fdiag{\disc(q)}. \]
  Thus $\alpha$ cannot be written as a finite combination of
  the $f_1^d$ (since the length of such a combination strictly decreases
  when applying $\Phi^+$). On the other hand, we may write it (at
  least formally for now) as
  \[ \alpha = \sum_{d\in \N}(-1)^d f_1^d. \] 
\end{ex}

But such an infinite combination may not always be well-defined:
since the $f_n^d$ take values in $A^{\pgq m}$ for increasing
values of $m$, any $\sum_{d\in \N}a_d f_n^d$ is well-defined
as an invariant with values in the completion of $A$ with respect
to its filtration, but usually not in $A$ itself, as the next example shows.

\begin{ex}\label{ex_moins_pfis}
  If $k$ is formally real, then $\sum_d f_1^d$
  sends $-\pfis{-1}$ to $\sum_{d\in \N} (-1)^d \{-1\}^d$,
  which is not in $A(k)$ (but is in its completion).
\end{ex}

It readily appears that the trouble is the bad behaviour of
the $f_n^d$ with respect to the opposites of Pfister forms.
To get a satisfying description of $M(n)$, we will introduce a new ``basis'',
with better balance between sums and differences of Pfister forms,
such that any infinite combination does take values in $A$.

\begin{defi}\label{def_g}
  Let $n\in \N^*$. For any $d\in \N$, we define $g_n^d\in M^{\pgq nd}(n)$ by:
  \begin{itemize}
    \bitem $g_n^0 = 1$ ;
    \bitem if $d\in \N^*$ is odd, $(g_n^d)^-=g_n^{d-1}$
    and $g_n^d(0)=0$ ;
    \bitem if $d\in \N^*$ is even, $(g_n^d)^+=g_n^{d-1}$
    and $g_n^d(0)=0$.
  \end{itemize}
  If $A=W$ (resp. $A=H$), we sometimes write $\gamma_n^d$ (resp. $v_{nd}^{(n)}$)
  for $g_n^d$.
\end{defi}

Corollary \ref{cor_phi_exact} ensures that these are well-defined.
This definition, which balances $\Phi^+$ and $\Phi^-$, gives
a reasonable behaviour under both operators:

\begin{prop}\label{prop_g_pm}
  Let $n\in \N^*$ and $d\in \N$. Then:
  \[ (g_n^{d+2})^{+-} = (g_n^{d+2})^{-+} = g_n^d ; \]
  \[ (g_n^{d+1})^+ = \left\{ \begin{array}{lc} g_n^d & \text{if $d$ odd} \\
             g_n^d + \{-1\}^n g_n^{d-1} & \text{if $d$ even ;} \end{array} \right. \]
  \[ (g_n^{d+1})^- = \left\{ \begin{array}{lc} g_n^d & \text{if $d$ even} \\
             g_n^d - \{-1\}^n g_n^{d-1} & \text{if $d$ odd.} \end{array} \right. \]
\end{prop}

\begin{proof}
  If $d$ is even, then $(g_n^{d+2})^-=g_n^{d+1}$ and $(g_n^{d+1})^+=g_n^d$,
  and if $d$ is odd, $(g_n^{d+2})^+=g_n^{d+1}$ and $(g_n^{d+1})^-=g_n^d$.
  In any case the first formula is satisfied.

  For the remaining two, we use $(g_n^{d+1})^+ - (g_n^{d+1})^- = \{-1\}^n g_n^{d-1}$
  coming from proposition \ref{prop_phi_pm}. We may conclude, arguing
  according to the parity of $d$.
\end{proof}

We may write the precise relation between $f_n^d$ and $g_n^d$:

\begin{prop}\label{prop_f_g}
  Let $n\in \N^*$. For any $d\in \N^*$:
  \begin{align*}
    g_n^d &= \sum_{k=\lfloor \frac{d}{2}\rfloor +1}^d
        \binom{\lfloor \frac{d-1}{2} \rfloor}{k-\lfloor \frac{d}{2}\rfloor -1} \{-1\}^{n(d-k)} f_n^k \\
    f_n^d &= \sum_{k=1}^d (-1)^{d-k}\binom{d-\lfloor \frac{k+1}{2}\rfloor -1}{\lfloor \frac{k}{2}\rfloor -1}
        \{-1\}^{n(d-k)} g_n^k.
  \end{align*}
  In particular, $(f_n^i)_{i\ppq d}$ and $(g_n^i)_{i\ppq d}$ generate the
  same submodule of $M(n)$.
\end{prop}

\begin{proof}
  Denote $\alpha_d$ the invariant defined by the right-hand side of
  the formula for $g_n^d$. If $d=2m$, the formula becomes
  \[ \alpha_d = \sum_{k=m+1}^{2m} \binom{m-1}{k-m-1} \{-1\}^{n(2m-k)} f_n^k\]
  which gives
  \[ \alpha_d^+ = \sum_{k=m+1}^{2m} \binom{m-1}{k-m-1} \{-1\}^{n(2m-k)} f_n^{k-1},\]
  and if $d=2m+1$ then we get
  \[ \alpha_d = \sum_{k=m+1}^{2m+1} \binom{m}{k-m-1} \{-1\}^{n(2m+1-k)} f_n^k\]
  hence
  \[ \alpha_d^+ = \sum_{k=m+1}^{2m+1} \binom{m}{k-m-1} \{-1\}^{2m+1-k} f_n^{k-1}.\]
  We thus have to check that in both cases we find the correct induction
  formula for $\alpha_{d+1}^+$ (coming from proposition \ref{prop_g_pm}).
  If $d=2m+1$, we have to show $\alpha_{2m+2}^+=\alpha_{2m+1}$, which is
  immediate given the above formulas. If $d=2m$, we have to show
  $\alpha_{2m+1}^+ = \alpha_{2m} + \{-1\}^n\alpha_{2m-1}$, so we need
  to compare
  \[ \sum_{k=m}^{2m}\binom{m}{k-m}\{-1\}^{n(2m-k)} f_n^k \]
  and
  \[ \sum_{k=m+1}^{2m}\binom{m-1}{k-m-1}\{-1\}^{n(2m-k)}f_n^k
         +  \sum_{k=m}^{2m-1}\binom{m-1}{k-m}\{-1\}^{n(2m-k)}f_n^k \]
  which are easily seen as being equal using Pascal's triangle.

  The formula for $f_n^d$ can be obtained either by inverting the one
  for $g_n^d$, or in a similar fashion. Let $\beta_d$ be the invariant
  defined by the right-hand side. Then we show that $\beta_d^+ = \beta_{d-1}$,
  separating the sums according to the parity of $k$:
  \begin{align*}
    \beta_d^+ =& (-1)^d\sum_m \binom{d-m-1}{m-1}\{-1\}^{n(d-2m)}(g_n^{2m})^+ \\
              &+ (-1)^{d+1}\sum_m \binom{d-m-2}{m-1}\{-1\}^{n(d-2m-1)}(g_n^{2m+1})^+ \\
              =& (-1)^d\sum_m \binom{d-m-1}{m-1}\{-1\}^{n(d-2m)}g_n^{2m-1} \\
              &+ (-1)^{d+1}\sum_m \binom{d-m-2}{m-1}\{-1\}^{n(d-2m-1)}(g_n^{2m}+\{-1\}^ng_n^{2m-1}) \\
              =& (-1)^{d+1}\sum_m \binom{d-m-2}{m-1}\{-1\}^{n(d-2m-1)}g_n^{2m} \\
              &+ (-1)^d\sum_m \left( \binom{d-m-1}{m-1} - \binom{d-m-2}{m-1}\right) \{-1\}^{n(d-2m)}g_n^{2m-1} \\
              =& (-1)^{d-1}\sum_m \binom{d-1-m-1}{m-1}\{-1\}^{n(d-1-2m)}g_n^{2m} \\
              &+ (-1)^{d-1+1}\sum_m \binom{d-m-1}{m-2} \{-1\}^{n(d-2m)}g_n^{2m-1}
  \end{align*}
  which does give $\alpha_{d-1}$.

  The last statement comes from the fact that the transition matrix from
  $(f_n^d)_d$ to $(g_n^d)_d$ is triangular unipotent.
\end{proof}

The important consequence of the balance of $g_n^d$ is given by:

\begin{prop}\label{prop_g_borne}
  Let $n\in \N^*$, and let $q\in I^n(K)$, such that
  $q=\sum_{i=1}^s \phi_i - \sum_{i=1}^t\psi_i$, where $\phi_i,\psi_i\in \Pf_n(K)$.
  Then for any $d>2\max(s,t)$, $g_n^d(q)=0$.
\end{prop}

\begin{proof}
  We may add hyperbolic forms in either sum so that $s=t$. Then we
  prove the statement by induction on $s$ : if $s=0$ then $q=0$,
  so for $d>0$ we have indeed $g_n^d(q)=0$ by construction.

  If the result holds up to $s-1$ for some $s\in \N^*$, then write
  $q' = q - \phi_s$ and $q'' = q' + \psi_s$. We get
  \begin{align*}
    g_n^d(q) =& g_n^d(q') + f_n(\phi_s)(g_n^d)^+(q') \\
             =& g_n^d(q'') -  f_n(\psi_s)(g_n^d)^-(q'') + f_n(\phi_s)(g_n^d)^+(q'') \\
             &- f_n(\phi_s)f_n(\psi_s)(g_n^d)^{+-}(q'')
  \end{align*}
  Now according to proposition \ref{prop_g_pm}, $(g_n^d)^-$, $(g_n^d)^+$
  and $(g_n^d)^{+-}$ may all be be expressed as combinations of some
  $g_n^k$ with $k\pgq d-2$, so we may apply the induction hypothesis
  with $q''$.
\end{proof}

\begin{coro}\label{cor_g_fixed_dim}
  If $q\in I(K)$ is the Witt class of an $r$-dimensional form,
  then $g_1^d(q)=0$ for any $d>r$.
\end{coro}

\begin{proof}
  Writing $r=2m$, if $q=\fdiag{a_1,b_1,\dots,a_m,b_m}$, then
  $q=\sum_{i=1}^m \pfis{-a_i}-\pfis{b_i}$, which allows to conclude
  using the previous proposition.
\end{proof}

Me may put it all together to prove the central theorem :

\begin{thm}\label{thm_g}
  Let $n\in \N^*$, and let $N(n)=A(k)^{\N}$, which is a filtered
  $A(k)$-module for the filtration
  $N^{\pgq m}(n) = \ens{(a_d)_{d\in \N}}{a_d\in A^{\pgq m-nd}}$.
  
  The following applications are mutually inverse isomorphisms
  of filtered $A(k)$-modules :
  \[ \isomdef{F}{N(n)}{M(n)}{(a_d)_{d\in \N}}{\sum_{d\in \N} a_d g_n^d,} \]
  \[ \isomdef{G}{M(n)}{N(n)}{\alpha}{(\alpha^{[d]}(0))_{d\in \N}.} \]
  where $\alpha^{[d]} = \alpha^{m+,m-}$ if $d=2m$, and
  $\alpha^{[d]} = \alpha^{(m+1)+,m-}$ if $d=2m+1$.
\end{thm}

\begin{proof}
  First, the application $F$ is well-defined, since according to
  proposition \ref{prop_g_borne}, for any fixed $q\in I^n(K)$
  we have $g_n^d(q)=0$ for large enough $d$. Then $F$ and $G$
  are clearly module morphisms, and the fact that
  they respect the filtrations is just a reformulation of the
  fact that $g_n^d$ takes values in $A^{\pgq nd}$, and that $\Phi_n^\eps$
  has degree $-n$. Let $\alpha = \sum_d a_d g_n^d$. 
  Using proposition \ref{prop_ker_phi}, we see that for any
  $r,s\in \N$ we can compute $\alpha^{r+,s-}(0)$ ignoring
  the terms for $d$ large enough. Thus it is easy to see from
  proposition \ref{prop_g_pm} that $a_{2m}=\alpha^{m+,m-}(0)$
  and $a_{2m+1}=\alpha^{(m+1)+,m-}(0)$, which shows that $G\circ F = \Id$.

  We now prove that $G$ is injective, which finishes the proof
  of the theorem. Let $\alpha\in \Ker(G)$, and let $d\in \N$.
  According to proposition \ref{prop_pi_gen}, and using the last
  statement of proposition \ref{prop_f_g}, we see that $\alpha$
  is congruent to some combination $\sum_{nk< d}a_k g_n^k$ modulo
  $M^{\pgq d}(n)$. Now the exact sequence in proposition \ref{prop_ker_phi}
  allows to see that $a_k \equiv \alpha^{[k]}(0)$ modulo $A^{\pgq d-nk}(k)$,
  so, since $\alpha^{[k]}(0)=0$, $a_k\in A^{\pgq d-nk}(k)$. This in turn implies that
  $\sum_{nk\ppq d}a_k g_n^k\in M^{\pgq d}(n)$, and thus $\alpha\in M^{\pgq d}(n)$.
  Since this is true for any $d\in \N$, we may conclude that $\alpha = 0$.  
\end{proof}

\begin{coro}
  Let $n\in \N^*$ and let $\eps=\pm 1$. There is an exact sequence
  of filtered $A(k)$-modules
  \[ 0 \to A(k) \To M(n) \xrightarrow{\Phi_n^\eps} M(n)[-n] \To 0. \]
\end{coro}

\begin{proof}
  Like for corollary \ref{cor_phi_exact}, the only thing left to
  prove after proposition \ref{prop_ker_phi} is the surjectivity
  of $\Phi_n^\eps$, but it is an easy consequence of theorem \ref{thm_g}.
\end{proof}

\begin{coro}
  Let $n\in \N^*$ and $\alpha\in M(n)$. There is a unique sequence $(a_d)_{d\in \N}$
  with $a_d\in A(k)$ such that for any $q\in I^n(K)$ the infinite sum
  $\sum_{d\in \N}a_df_n^d(q)$ exists in $A(K)$ and is equal to $\alpha(q)$.
  Furthermore, for all $d\in \N$, $a_d = \alpha^{d+}(0)$.
\end{coro}

\begin{proof}
  If such a sequence exists, then using proposition \ref{prop_ker_phi}
  we find that $\alpha^{i+}(0)\equiv a_i$ modulo $A^{\pgq dn}(k)$
  for all $i\ppq d$, so for a fixed $i$ we can make $d$ go to
  infinity, and we find that indeed $a_d = \alpha^{d+}(0)$, which
  shows uniqueness.
  
  For existence, write $\alpha = \sum_d b_d g_n^d$, and decompose each
  $g_n^d$ in terms of the $f_n^i$ using proposition \ref{prop_f_g}.
  Then we find a decomposition of $\alpha$ in terms of $f_n^d$ which
  is valid pointwise, and the $a_d$ we find are well-defined in $A(k)$ since
  each $a_d$ is a combination of a finite number of $b_i$ (using that
  $f_n^i$ appears appears in the decomposition of $g_n^d$ only if
  $d\ppq 2i$).
\end{proof}

\begin{rem}
  In particular, any invariant of $I^n$ with values in
  $H^d(-,\mu_2)$ may be lifted to an invariant with values in $I^d$.
\end{rem}

\begin{rem}
  If $k$ is not a formally real field, then for large enough $d$
  we have $\{-1\}^d=0$, and thus according to formula (\ref{eq_fn_moins_phi})
  $f_n^d(-\phi)=0$ for any $\phi\in \Pf_n(K)$. This implies
  that in this case, for any $q\in I^n(K)$ we have $f_n^d(q)=0$
  for large enough $d$ (for the same reasons as in corollary
  \ref{cor_sum_f}), and so we may use the $f_n^d$ instead
  of the $g_n^d$ in the theorem (with $G(\alpha) = (\alpha^{d+}(0))_d$).
  In the extreme case where $-1$
  is a square in $k$, we actually even get $f_n^d=g_n^d$, as
  can be seen from proposition \ref{prop_f_g}.
  On the other hand, example \ref{ex_moins_pfis} shows that we
  cannot use the $f_n^d$ if $k$ is formally real. What happens
  in this case is that an arbitrary infinite combination of the
  $f_n^d$ does correspond to a combination of the $g_n^d$ (using
  proposition \ref{prop_f_g}), but with coefficients in the completion
  of $A$ with respect to its filtration.
\end{rem}

\begin{rem}
  We may construct cohomological invariants $\alpha$ such that, even
  though the degree of $\alpha(q)$ is bounded for fixed $q$, it is
  unbounded when $q$ varies (for instance, $\alpha = \sum_d g_n^d$).
  This reflects in some sense the ``infinite'' nature of $I^n$,
  and it is a behaviour that does not appear for invariants of algebraic
  groups. The submodule $M'(n)$ of uniformly bounded cohomological invariant
  is the submodule generated by the $f_n^d$ (or by the $g_n^d$). We may
  write that $M(n) = \Inv\left( I^n,\varinjlim H^{\ppq d}(-,\mu_2)\right)$,
  while $M'(n) = \varinjlim \Inv\left( I^n, H^{\ppq d}(-,\mu_2)\right)$.
\end{rem}

\section{Algebra structure}

Since $M(n)$ is not only a $A(k)$-module, but also
an algebra, we wish to understand how the product can be expressed
in terms of the basic elements $f_n^d$.

For this section, if $d\in \N$, and $p,q\in \N$ are such that $p+q\ppq d$, we set
\[ C^d_{p,q}= \frac{d!}{p!\cdot q!\cdot (d-p-q)!}. \]
This is just a more compact notation for the usual multinomial $\binom{d}{p,q,d\minus p\minus q}$.

\begin{prop}\label{prop_prod_phi}
  Let $n\in \N^*$, and $\eps=\pm 1$. Then for any $\alpha,\beta\in M(n)$:
  \[ \Phi^\eps(\alpha\beta) = \Phi^\eps(\alpha)\beta + \alpha\Phi^\eps(\beta)
             + \eps\{-1\}^n\Phi^\eps(\alpha)\Phi\eps(\beta). \]
\end{prop}

\begin{proof}
  Let $q\in I^n(K)$ and $\phi\in \Pf_n(K)$. Then
  \begin{align*}
    (\alpha\beta)(q+\eps\phi)
       &= (\alpha(q)+\eps f_n(\phi)\alpha^\eps(q))\cdot (\beta(q)+\eps f_n(\phi)\beta^\eps(q)) \\
       &= (\alpha\beta)(q)
           + \eps f_n(\phi)\left( (\alpha^\eps\beta)(q) + (\alpha\beta^\eps)(q)
           + \eps \{-1\}^n(\alpha^\eps\beta^\eps)(q)\right). \qedhere
  \end{align*}
\end{proof}

\begin{prop}\label{prop_prod}
  Let $n\in \N^*$ and $s,t\in \N$. Then 
  \[ f_n^s \cdot f_n^t = \sum_{d=\max(s,t)}^{s+t} C^d_{d-s,d-t} \{-1\}^{n(s+t-d)} f_n^d. \]
\end{prop}

\begin{proof}
  First note that both side of the equality have the same value in $0$ (which is
  $1$ if $s=t=0$ and $0$ otherwise). So we just need to show that applying $\Phi$
  to both sides of the equation gives the same expression.

  Now proposition \ref{prop_prod_phi} gives:
  \begin{equation}\label{eq_prod}
    \Phi\left( f_n^s \cdot f_n^t\right) = f_n^s\cdot f_n^{t-1} + f_n^{s-1}\cdot f_n^t
         + \{-1\}^n f_n^{s-1}\cdot f_n^{t-1}.  
  \end{equation}

  We proceed by induction, say on $(s,t)$ with lexicographical order.
  First the result is clear if $s=0$ or $t=0$. 

  Then by induction we can replace each term in (\ref{eq_prod}) and rearrange
  them to find for $\Phi\left( f_n^s\cdot f_n^t\right)$:
  \begin{equation}\label{eq_prod_2}
    \binom{s}{t} \{-1\}^{nt} f_n^{s-1} +  \binom{s+t}{t} f_n^{s+t-1}
          +  \sum_{d=s}^{s+t-2} C^{d+1}_{d-s+1,d-t+1} \{-1\}^{n(s+t-d-1)} f_n^d
  \end{equation}
  where for the coefficient before $f_n^{s-1}$ we use $\binom{s-1}{t} + \binom{s-1}{t-1} = \binom{s}{t}$,
  for that of $f_n^{s+t-1}$ we use
  $\binom{s+t-1}{t} + \binom{s+t-1}{t-1} = \binom{s+t}{t}$, and for the other terms we use
  \begin{equation*}
    C^d_{d-s+1,d-t} + C^d_{d-s,d-t+1} + C^d_{d-s+1,d-t+1} = C^{d+1}_{d-s+1,d-t+1}.
  \end{equation*}

  We can them compute that applying $\Phi$ to the right-hand side of the equality
  in the statement of the proposition yields exaclty (\ref{eq_prod_2}).
\end{proof}

Of course there is a corresponding formula for the products of the $g_n^d$,
but it turns out that it is much more involved, and we will not address
it here. This means that although we have a nice module isomorphism between
$M(n)$ and $A(k)^\N$, transporting the algebra structure of $M(n)$ to
$A(k)^\N$ is not as convenient. On the other hand, if we use the $f_n^d$
we only have a module isomorphism between $M(n)$ and a submodule of $A(k)^\N$
which is hard to describe, but we can transport the product in reasonably
easy way.

There are several cases where the formula of proposition \ref{prop_prod}
can be greatly simplified by studying the parity of the multinomials that appear.
We introduce some notations: if $s,t\in \N$, we write $s\lor t$
(resp. $s\land t$) for the integer obtained by applying a bitwise \emph{or}
(resp. a bitwise \emph{and}) to the binary representations of $s$ and $t$.
In particular $s\lor t +s\land t = s+t$.

\begin{lem}\label{lem_binom}
  Let $d\in \N$, and $s,t\in \N$ such that $\max(s,t)\ppq d\ppq s+t$. Then
  $C^d_{d-s,d-t}$ is odd iff $d=s\lor t$.
\end{lem}

\begin{proof}
  It is well-known that for any $a\in \Z$, the $2$-adic valuation of $a!$
  is $a-f(a)$ where $f(a)$ is the number of $1$s in the binary representation
  of $a$. Then :
  \begin{align*}
    v_2\left( C^d_{d-s,d-t} \right) =& (d-f(d)) - (s+t-d -f(s+t-d)) \\
        &- (d-s - f(d-s)) - (d-t - f(d-t)) \\
       =&  f(s+t-d) + f(d-s) + f(d-t) - f(d).
  \end{align*}

  But it is easily seen that for any $a,b\in \Z$, $f(a+b)\ppq f(a)+f(b)$,
  with equality iff $a\land b =0$. Thus $C^d_{d-s,d-t}$ is odd
  iff $s+t-d$, $d-s$ and $d-t$ have pairwise disjoint binary representations.

  We claim this is equivalent to $d=s \lor t$. Indeed, if $d=s\lor t$
  it is obvious, and if $d\neq s\lor t$, consider the weakest bit where
  $d$ and $s\lor t$ differ; there are several possibilities
  for the bits of $s$, $t$ and $d$ in this slot: 
  $s$ has 1 and $d$ has 0, $t$ has 1 and $d$ has 0, or $s$ and $t$ have 0
  and $d$ has 1. In all these cases, at least two numbers among $d-s$, $d-t$
  and $s+t-d$ have a 1 in this slot, and their binary representations are thus
  not disjoint.
\end{proof}

Then we can state:

\begin{coro}\label{cor_prod_cohom}
  Let $n\in \N^*$ and $s,t\in \N$. If $A = H$, then :
  \[ u_{ns}^{(n)} \cup u_{nt}^{(n)} = (-1)^{n(s\land t)}\cup u_{n(s\lor t)}^{(n)}. \]
\end{coro}

\begin{proof}
  Since $H^*(k,\mu_2)$ is a ring of characteristic $2$, using lemma \ref{lem_binom},
  we see that the only potentially non-zero term in the formula of proposition \ref{prop_prod}
  is $\{-1\}^{s\land t}f_n^{s\lor t}$.
\end{proof}

\begin{rem}
  This is very reminiscent of the formula for the product of Stiefel-Whitney
  classes, since $w_s\cup w_t= (-1)^{s\land t}\cup w_{s\lor t}$. When $-1$ is a square,
  this is easily explained by the fact that $u^{(1)}_d$ coincides with the Stiefel-Whitney
  map $w_d$ (see remark \ref{rem_sw_moins_1}), but in general $w_d$ is not well-defined
  on Witt classes so the formulas are really different phenomena.
\end{rem}

\begin{coro}\label{cor_prod_moins_1}
  Let $n\in \N^*$ and $s,t\in \N$. If $-1$ is a square in $k$, then
  $f_n^s\cdot f_n^t$ equals $f_n^{s+t}$ if $s\land t = 0$, and $0$ otherwise.
\end{coro}

\begin{proof}
  Note that in this situation $A(k)$ is also a ring of characteristic $2$,
  so the same reasoning as in corollary \ref{cor_prod_cohom} applies,
  but this time if $s\land t\neq 0$ the term is also 0.
\end{proof}

\begin{rem}
  Consider the case $A=H$, and the submodule $M'(n)\subset M(n)$ generated
  by the $u_{nd}^{(n)}$, which is the subalgebra of cohomological invariants with
  uniformly bounded degree. Then from corollary \ref{cor_prod_cohom} we find a very
  simple algebra presentation of $M'(n)$: the (commuting) generators are $x_i=u_{n2^i}^{(n)}$,
  and the relations are given by $x_i^2 = \{-1\}^{n2^i}x_i$.
\end{rem}

\section{Restriction from $I^n$ to $I^{n+1}$}

For any $m,n\in \N^*$ with $m\pgq n$, there is an obvious restriction morphism
\begin{equation}\label{eq_rho}
  \foncdef{\rho_{n,m}}{M(n)}{M(m)}{\alpha}{\alpha_{|I^m}.} 
\end{equation}

Given the definition of $f_n^d$, if we want to express $(f_n^d)_{|I^{n+1}}$
in terms of the $f_{n+1}^k$, it is natural to try to express $\pi_n^d$
in terms of the $\pi_{n+1}^k$ in $GW(K)$.

\begin{prop}\label{prop_restr_pi}
  Let $n\in \N^*$. For any $d\in \N^*$, we have 
  \[ \pi_n^d = \sum_{\frac{d}{2}\ppq k\ppq d} \binom{k}{d-k} 2^{(d-k)(n-1)} \pi_{n+1}^k. \]
\end{prop}

\begin{proof}
  We define $p_n(t) = t + 2^{n-1}t^2\in \Z[t]$. Then recall that $(\pi_n)_t=\lambda_{h_n(t)}$
  where $h_n = x_n^{\circ -1}$, and $x_n$ is defined recursively by $x_{n+1} = p_n\circ x_n$.
  Thus we have the formula $h_n = h_{n+1}\circ p_n$, and
  \[ (\pi_n)_t = (\pi_{n+1})_{p_n(t)}. \]
  Therefore we find
  \begin{align*}
    \sum_d \pi_n^d\cdot t^d &= \sum_k \pi_{n+1}^k(t+2^{n-1}t^2)^k \\
      &= \sum_k \sum_{k\ppq d\ppq 2k} \binom{k}{d-k}2^{(d-k)(n-1)}\pi_{n+1}^k\cdot t^d
  \end{align*}
  which gives the result.
\end{proof}

Then we deduce the corresponding results for our invariants.

\begin{coro}\label{cor_restr_w}
  Let $n,d\in \N^*$. If $A=W$ then
  \[ (\bar{\pi}_n^d)_{|I^{n+1}}
         = \sum_{\frac{d}{2}\ppq k\ppq d} \binom{k}{d-k}\pfis{-1}^{(d-k)(n-1)} \bar{\pi}_{n+1}^k.  \]
\end{coro}

\begin{proof}
  This is an immediate consequence of the proposition, given that in
  $W(K)$ we have $\pfis{-1}=2$.
\end{proof}

\begin{coro}\label{cor_restr_cohom}
  Let $n,d\in \N^*$. If $A=H$ then
  \[   (u_{nd}^{(n)})_{|I^{n+1}} = \left\{ \begin{array}{lc}
                                      (-1)^{m(n-1)}\cup u_{(n+1)m}^{(n+1)} & \text{if $d=2m$} \\
                                      0 & \text{if $d$ odd}
                                    \end{array} \right.     \]
\end{coro}

\begin{proof}
  This is also a consequence of the proposition, but we have to notice that
  when we apply $e_{nd}$ to the formula, the terms corresponding to $k>d/2$
  vanish. Indeed, in this case $\pfis{-1}^{(d-k)(n-1)} \pi_{n+1}^k$ sends
  $\hat{I}^{n+1}(K)$ to $\hat{I}^r(K)$ with $r=(d-k)(n-1)+k(n+1) = d(n-1) +2k > nd$.
  Thus composing with $e_{nd}$ will give zero.

  So only the term $k=d/2$ remains (and only when $d$ is even). 
\end{proof}

\begin{rem}\label{rem_restr_cohom}
  In particular, for cohomological invariants, and when $n=1$,
  we get the simple formula: $(u_{2d}^{(1)})_{|I^2}=u_{2d}^{(2)}$, which shows
  that any cohomological invariant of $I^2$ extends (not uniquely)
  to $I$. On the other hand, for $n\pgq 3$ and $d\pgq 1$, $u_{nd}^{(n)}$ never extends
  to $I^{n-1}$. This vastly generalizes the familiar facts that $e_2$
  extends to $I$, but $e_3$ does not extend to $I^2$.
\end{rem}

\begin{rem}
  Suppose $-1$ is a square in $k$, and take $n\pgq 2$.
  Then in the case of Witt invariants, $\bar{\pi}_n^d$ is independent
  of $n$, and in the case of cohomological invariants the restriction
  of any $\alpha\in M(n)$ to $I^{n+1}$ is constant.
\end{rem}

As an application of corollary \ref{cor_restr_cohom}, we may improve a result of
Kahn in \cite{Kah}: he shows in the proof of proposition 3.3 that if
$H^r(K,\mu_2)$ has symbol length at most $l\in \N$, then any element
of $H^{r(l+1)}(K,\mu_2)$ is a multiple of $(-1)\in H^1(K,\mu_2)$.
We would like to thank Karim Becher for fruitful discussions
about this application during a visit in Antwerp.

\begin{prop}
  Let $r\in \N^*$, and assume that $H^r(K,\mu_2)$ has symbol length
  at most $l\in \N$. Then for any $d> l$, we have
  \[ H^{rd}(K,\mu_2) \subset (-1)^{(r-1)\left\lceil \frac{d-l}{2} \right\rceil}\cup H^*(K,\mu_2). \]
  In particular, any element of $H^m(K,\mu_2)$ for $m\pgq r(l+1)$ is
  a multiple of $(-1)^{r-1}\in H^{r-1}(K,\mu_2)$.
\end{prop}

\begin{proof}
  It is enough to prove the result for Galois symbols: let $\alpha\in H^{rd}(K,\mu_2)$
  be a symbol, and write $\alpha = \alpha_1\cup\cdots\cup \alpha_d$ with
  $\alpha_i\in H^r(K,\mu_2)$. Then we set $\phi_i\in \Pf_r(K)$ such that
  $e_r(\phi_i)=\alpha_i$, and $q=\sum_i \phi_i\in I^r(K)$. According to
  corollary \ref{cor_sum_f}, we have $\alpha = u_{rd}^{(r)}(q)$.

  Now by hypothesis, $q = q' + x$ where $q'\in I^r(K)$ can be written as a
  sum of $l$ or less $r$-fold Pfister forms, and $x\in I^{r+1}(K)$. We have
  \begin{align*}
    \alpha &= u^{(r)}_{rd}(q' + x) \\
    &= \sum_{k=0}^{rd} u^{(r)}_{rk}(q')\cup u^{(r)}_{r(d-k)}(x).
  \end{align*}
  But corollary \ref{cor_sum_f} shows that $u^{(r)}_{rk}(q')=0$
  when $k>l$, and corollary \ref{cor_restr_cohom} shows that
  $u^{(r)}_{r(d-k)}(x)$ is a multiple of $(-1)^{(r-1)\left\lceil \frac{d-k}{2}\right\rceil}$.
  In the end, $\alpha$ is a multiple of $(-1)^{(r-1)\left\lceil \frac{d-l}{2} \right\rceil}$.
\end{proof}

\section{Similitudes}

In this section we study the behaviour of invariants with respect
to similitudes.

\begin{propdef}\label{def_psi}
  There is a unique morphism of filtered $A(k)$-modules
  \[ \begin{foncdef}{\Psi}{\Inv(W,A)}{\Inv_0(W,A)[-1]}{\alpha}{\tld{\alpha}}\end{foncdef}  \]
  such that
  \begin{equation}
    \alpha(\fdiag{\lambda}q) = \alpha(q) + \{\lambda\}\tld{\alpha}(q)
  \end{equation}
  for any $\alpha\in \Inv(W,A)$, $q\in F(K)$ and $\lambda\in K^*$. 

  If $F$ be a subfunctor of $W$ such that $F(L)$ is stable under similitudes
  for any $L/k$, and $0\in F(k)$, then $\Psi$ restricts to a morphism
  $\Inv(F,A)\to \Inv_0(F,A)[-1]$. In particular, for any $n\in \N^*$
  we get a filtered morphism $M(n)\o M_0(n)[-1]$.
\end{propdef}

\begin{proof}
  Let $\alpha\in \Inv(F,A^{\pgq d})$ for some $d\in \N$ and $q\in F(K)$.
  For any $\lambda\in L^*$, where $L/K$ is any field extension, we set
  $\beta_q(\lambda) = \alpha(\fdiag{\lambda} q)$.

  Then $\beta_q$ is an invariant over $K$ of square classes, with values in $A$.
  Now the functor of square classes is isomorphic to $\Pf_1$, so we may apply
  lemma \ref{lem_a_inv}: there are uniquely determined $x_q,y_q\in A(K)$
  such that $\beta_q(\lambda) = x_q + \{\lambda\}\cdot y_q$
  for all $\lambda$.
  Taking $\lambda = 1$ we see that $x_q = \alpha(q)$, and we set $\tld{\alpha}(q) = y_q$.

  The uniqueness of $y_q$ allows to see that $\tld{\alpha}\in \Inv(F,A)$. The
  fact that $\{\lambda\}\cdot y_q\in A^{\pgq d}(L)$ for all $\lambda\in L^*$
  shows according to lemma \ref{lem_a_filtr} that $y_q\in A^{\pgq d-1}(K)$,
  so as a filtered morphism $\Psi$ has degree $-1$. Finally, it is
  clear that if $q=0$, then $\alpha(\fdiag{\lambda}q) = \alpha(q) + \{\lambda\}\cdot 0$,
  so $\tld{\alpha}(0)=0$, which means that $\tld{\alpha}$ is normalized.
\end{proof}

We first establish some basic properties of $\Psi$:

\begin{prop}
  Let $\alpha,\beta\in \Inv(W,A)$. Then
  \[ \Psi(\alpha\beta) = \Psi(\alpha)\beta + \alpha\Psi(\beta) + \{-1\}\Psi(\alpha)\Psi(\beta). \]
\end{prop}

\begin{proof}
  Let $q\in W(K)$ and $\lambda\in K^*$. Then:
  \begin{align*}
    (\alpha\beta)(\fdiag{\lambda}q)
      &= (\alpha(q) + \{\lambda\}\tld{\alpha}(q))(\beta(q) + \{\lambda\}\tld{\beta}(q)) \\
      &= (\alpha\beta)(q) + \{\lambda\}\left( (\tld{\alpha}\beta)(q) + (\alpha\tld{\beta})(q)
             + \{-1\}(\tld{\alpha}\tld{\beta})(q) \right). \qedhere
  \end{align*}
\end{proof}

\begin{prop}
  We have $\Psi^2 = -\delta(A) \Psi$.
\end{prop}

\begin{proof}
  For any extension $L/K$ and any $\lambda,\mu\in L^*$ :
  \begin{align*}
    \alpha(\fdiag{\lambda\mu}q) &= \alpha(\fdiag{\lambda}q)
                                  + \{\mu\} \tld{\alpha}(\fdiag{\lambda}q) \\
    &= \alpha(q) + \{\lambda\} \tld{\alpha}(q) + \{\mu\} \tld{\alpha}(q)
      + \{\lambda,\mu\} \tld{\tld{\alpha}}(q) \\
    &= \alpha(q) + \{\lambda\mu\} \tld{\alpha}(q) + \{\lambda,\mu\}\left(\delta \tld{\alpha}(q) + \tld{\tld{\alpha}}(q) \right)
  \end{align*}
  using formula (\ref{eq_delta_sum}) for the last equality.
  We also have
  \[ \alpha(\fdiag{\lambda\mu}q) = \alpha(q) + \{\lambda\mu\} \tld{\alpha}(q), \]
  so $\{\lambda,\mu\}\left(\delta \tld{\alpha}(q) + \tld{\tld{\alpha}}(q) \right)=0$.
  Since this holds for any $\lambda$, $\mu$ over any extension, we may conclude
  that $\tld{\tld{\alpha}}(q)= -\delta \tld{\alpha}(q)$.
\end{proof}

\begin{rem}
  By definition, $\tld{\alpha}=0$ iff $\alpha(\fdiag{\lambda}q)=\alpha(q)$,
  that is to say $\alpha$ is \emph{invariant under similitudes}. But the previous
  proposition suggests that in the case $A=W$, $\tld{\alpha}=-\alpha$
  should also be an interesting property (notably, it is always satisfied
  by invariants of the form $\tld{\beta}$). And indeed, it is easily seen
  to be equivalent to $\alpha(\fdiag{\lambda}q)=\fdiag{\lambda}\alpha(q)$,
  in which case we say $\alpha$ is \emph{compatible with similitudes}. Then the proposition
  shows that any $\alpha$ may be uniquely decomposed as a sum $\alpha = \beta + \gamma$
  with $\beta$ compatible with similitudes, and $\gamma$ invariant under similitudes.
  Precisely: $\beta = -\tld{\alpha}$ and $\gamma = \alpha + \tld{\alpha}$.

  From a less intrinsic point of view, if $\alpha$ is a finite combination
  of the $f_n^d$, then by definition of the $f_n^d$ it can be seen as a
  composition
  \[ I^n(K) \Isom \hat{I}^n(K) \subset GW(K) \xrightarrow{h} GW(K) \To W(K) \]
  where $h$ is a combination of the $\lambda^i$. Then $\beta$ corresponds to
  selecting only the odd $i$, while $\gamma$ corresponds to the even terms.
  Thus it makes sense to call $\beta$ the \emph{odd part} of $\alpha$, and
  $\gamma$ its \emph{even part}. This decomposition has no clear equivalent
  for cohomological invariants.
\end{rem}

We now want to describe the action of $\Psi$ on our basic invariants. It
turns out that it is much easier to deal with the $g_n^d$ than the $f_n^d$ in
this situation.

\begin{prop}\label{prop_simil}
  Let $n,d\in \N^*$. Then
  \[ \tld{g_n^d} = \left\{ \begin{array}{lc}
                             -\delta(A) g_n^d & \text{if $d$ odd} \\
                             \{-1\}^{n-1}g_n^{d-1} & \text{if $d$ even.}
                           \end{array} \right. \]
\end{prop}

\begin{proof}
  We prove the proposition by induction on $d$. If $d=1$, the statement means that
  \[ f_n(\fdiag{\lambda}q) = f_n(q) -\delta \{\lambda\} f_n(q), \]
  which is true whether $A=W$ or $A=H$.

  Now suppose the proposition holds until $d-1$, for some $d\pgq 2$.
  Since $\tld{g_n^d}$ is normalized, it is enough to compute $\tld{g_n^d}^+$.
  Let $L/K$ be any extension, and take $q\in I^n(K)$, $\phi\in \Pf_n(L)$
  and $\lambda\in L^*$. Then:
  \begin{align*}
    g_n^d(\fdiag{\lambda}(q+\phi)) &= g_n^d(q+\phi) + \{\lambda\}\tld{g_n^d}(q+\phi) \\
                                   &= g_n^d(q) + f_n(\phi)(g_n^d)^+(q) + \{\lambda\} \tld{g_n^d}(q) +
                                     \{\lambda\}f_n(\phi) \tld{g_n^d}^+(q)
  \end{align*}
  so if we consider generic $\lambda$ and $\phi$ and take residues, we
  find exactly $\tld{g_n^d}^+(q)$.
  
  On the other hand, if we write $\phi = \pfis{a}\psi$, we can compute
  \[ g_n^d(\fdiag{\lambda}(q+\phi)) = g_n^d(\fdiag{\lambda}q + \pfis{\lambda a}\psi - \pfis{\lambda}\psi) \]
  using succesively on each term $\Phi^+$ relative to $\pfis{\lambda a}$,
  $\Phi^-$ relative to $\pfis{\lambda}$, and $\Psi$ relative to $\fdiag{\lambda}$,
  to get a 8-term sum. Again considering generic $\lambda$, $a$ and $\psi$,
  taking residues, and comparing to the previous computation, we find:
  \begin{equation}\label{eq_rec_simil}
    \tld{g_n^d}^+ = -\delta (g_n^d)^+  - \tld{(g_n^d)^+} + \{-1\}^{n-1}(g_n^d)^{+-} + \{-1\}^n\tld{(g_n^d)^{+-}},
  \end{equation}
  using several times equations (\ref{eq_delta_sum}) and (\ref{eq_delta_2}).
  
  If $d$ is even, then $(g_n^d)^+=g_n^{d-1}$ and $(g_n^d)^{+-}=g_n^{d-2}$,
  so by induction $\tld{(g_n^d)^+} = -\delta g_n^{d-1}$ and
  $\tld{(g_n^d)^{+-}} = \{-1\}^{n-1}g_n^{d-3}$.
  Thus from equation (\ref{eq_rec_simil}) we get:
  \begin{align*}
    \tld{g_n^d}^+ &= -\delta g_n^{d-1} + \delta g_n^{d-1} + \{-1\}^{n-1}(g_n^{d-2}+\{-1\}^ng_n^{d-3}) \\
                  &= \{-1\}^{n-1}(g_n^{d-1})^+
  \end{align*}
  which is the expected formula (we need to be a little careful with the
  case $d=2$, but we can check that the reasoning still holds if we say
  that $g_n^{-1}=0$).

  Similarly, if $d$ is odd, $(g_n^d)^+=g_n^{d-1} + \{-1\}^ng_n^{d-2}$
  and $(g_n^d)^{+-}=g_n^{d-2}$, so
  $\tld{(g_n^d)^+} = \{-1\}^{n-1}g_n^{d-2} - \delta\{-1\}^ng_n^{d-2} = -\{-1\}^{n-1}g_n^{d-2}$,
  and $\tld{(g_n^d)^{+-}} = -\delta g_n^{d-2}$. Then from (\ref{eq_rec_simil}):
  \begin{align*}
    \tld{g_n^d}^+ &= -\delta (g_n^d)^+ + \{-1\}^{n-1} g_n^{d-2}  + \{-1\}^{n-1}g_n^{d-2} -\delta \{-1\}^n g_n^{d-2} \\
                  &= -\delta (g_n^d)^+
  \end{align*}
 using (\ref{eq_delta_2}), which also allows to conclude.
\end{proof}

\begin{coro}\label{cor_inv_sim}
  The module $\Inv(I^n/\sim, A)$ of invariants of similarity classes
  of elements in $I^n$ is given by the combinations $\sum_{d\in \N} a_d g_n^d$
  with $\{-1\}^{n-1}a_{2i+2} = \delta(A)a_{2i+1}$ for all $i\in \N$.
\end{coro}

\begin{proof}
  The module $\Inv(I^n/\sim, A)$ is naturally isomorphic to the kernel
  of $\Psi$, and if $\alpha = \sum_{d\in \N}a_dg_n^d$, we get
  \[ \tld{\alpha} = \sum_{i\in \N} (\{-1\}^{n-1}a_{2i+2}-\delta(A)a_{2i+1}) g_n^{2i+1}, \]
  which gives the result.
\end{proof}

The formula for $\tld{f_n^d}$ is not particularly enlightening (see remark
\ref{rem_f_simil}), but we may at least give the values of $f_n^d$ on general
Pfister forms (which amounts to computing the values of $\tld{f_n^d}$ on Pfister
forms). This may be deduced from the previous proposition using \ref{prop_f_g},
but we can give a direct proof.

\begin{prop}\label{prop_f_gen_pfis}
  Let $n\in \N^*$ and $d\pgq 2$. Then for any $\phi \in \Pf_n(K)$
  and $\lambda\in K^*$ we have
  \[ f_n^d(\fdiag{\lambda}\phi) = (-1)^d \{-1\}^{n(d-1)-1}\{\lambda\}f_n(\phi). \]
\end{prop}

\begin{proof}
  Write $\phi = \pfis{x}\psi$. Then since
  $\fdiag{\lambda}\pfis{x} = \pfis{\lambda x} - \pfis{\lambda}$,
  and using formula (\ref{eq_fn_moins_phi}), we get :
  \begin{align*}
    f_n^d(\fdiag{\lambda}\phi) =& f_n^d(\pfis{\lambda x}\psi - \pfis{\lambda}\psi) \\
       =& f_n^d(-\pfis{\lambda}\psi) + \{\lambda x\} f_n^{d-1}(-\pfis{\lambda}\psi) \\
       =& (-1)^d \{-1\}^{n(d-1)} \{\lambda\}f_{n-1}(\psi) \\
        &+ \{\lambda x\}f_{n-1}(\psi) (-1)^{n(d-1)}\{-1\}^{n(d-2)}\{\lambda\}f_{n-1}(\psi) \\
       =& (-1)^d\{-1\}^{n(d-1)-1}\left( \{-1\} \{\lambda\} - \{\lambda x\}\{\lambda\}\right) f_{n-1}(\psi) \\
       =& (-1)^d\{-1\}^{n(d-1)-1}\{\lambda\}\{x\}f_{n-1}(\psi). \qedhere
  \end{align*}
\end{proof}

\begin{rem}\label{rem_f_simil}
  We can give the general formula for $\tld{f_n^d}$ for the record,
  though we will not prove it:
  \[ \tld{f_n^d} = (-1)^d\sum_{k=1}^{d-1}\binom{d-1}{k-1} \{-1\}^{n(d-k)-1} f_n^k
         + \left\{ \begin{array}{cl} 0 & \text{if $d$ even} \\
                     -\delta(A) f_n^d & \text{if $d$ odd.}\end{array} \right. \]
  We can check that if we evaluate this on a Pfister form we retrieve
  proposition \ref{prop_f_gen_pfis}, and as an even more special case
  formula (\ref{eq_fn_moins_phi}).
\end{rem}

\section{Ramification of invariants}

In this short section we establish the behaviour of invariants with
respect to residues of discrete valuations (which incidentally was one
of the main initial motivations of this article). Let thus $(K,v)$
be a valued field, where $v$ is a rank 1 discrete $k$-valuation,
with valuation ring $\mathcal{O}_K$ and residue field $\kappa$
(in particular, $\kappa$ is an extension of $k$, so it has characteristic
not 2).

Recall from \cite[19.10]{EKM} the so-called second residue map
$\partial_\pi: W(K)\to W(\kappa)$, which depends on the choice of a
uniformizing element $\pi\in K$. We say that $q\in W(K)$ is
\emph{unramified} if $\partial_\pi(q)=0$, which is independent
of the choice of $\pi$. Then $q$ is unramified iff it has a
diagonalization $\fdiag{a_1,\dots,a_r}$ with $a_i\in \mathcal{O}_K^*$.

Recall also from \cite[7.9]{GMS} the canonical residue map
$\partial: H^d(K,\mu_2)\to H^{d-1}(\kappa,\mu_2)$, which extends
to $\partial: H^*(K,\mu_2)\to H^*(\kappa,\mu_2)$.
We say that $x\in H^*(K,\mu_2)$ is unramified if $\partial(x)=0$.

Furthermore, from \cite[19.14]{EKM}, we have
$\partial_\pi(I^d(K))\subset I^{d-1}(\kappa)$, and using for instance
\cite[101.8]{EKM} we get for any $d\in \N^*$ a commutative diagram
\[ \begin{tikzcd}
  I^d(K) \rar{\partial_\pi} \dar{e_n} & I^{d-1}(\kappa) \dar{e_{d-1}} \\
  H^d(K,\mu_2) \rar{\partial} & H^{d-1}(\kappa,\mu_2).
\end{tikzcd} \]

\begin{prop}\label{prop_ram}
  Let $n\in \N^*$ and $q\in I^n(K)$, where $K$ is endowed with a
  rank 1 discrete $k$-valuation. If $q$ is unramified,
  then $\alpha(q)\in A(K)$ is unramified for any $\alpha\in M(n)$.
\end{prop}

\begin{proof}
  By hypothesis, $\hat{q}\in \hat{I}^n(K)$ comes from an element
  of $GW(\mathcal{O}_K)$, so any $\lambda^i(\hat{q})$ also comes
  from $GW(\mathcal{O}_K)$, and is unramified. Since $\pi_n^d$
  is a combination of the $\lambda^i$ with integer coefficients,
  $\pi_n^d(\hat{q})\in \hat{I}^{nd}(K)$ is unramified.

  Now tautologically if $A=W$, and applying the above commutative
  diagram if $A=H$, this implies that $f_n^d(q)\in A^{\pgq nd}(K)$
  is unramified.

  Since any $\alpha\in M(n)$ is a combination of the $f_n^d$ with
  coefficients in $A(k)$, and $v$ is a $k$-valuation, we can conclude
  that $\alpha(q)\in A(K)$ is unramified.
\end{proof}

\section{Invariants of $\mathbf{Quad_{2r}}$}

In \cite{GMS}, Serre gives a complete description of $\Inv(\Quad_m, A)$:
it is a free $A(k)$-module of rank $n+1$, with basis $(\lambda^d)_{0\ppq d\ppq m}$
for $A=W$, and the Stiefel-Whitney classes $(w_d)_{0\ppq d\ppq m}$ for $A=H$
(see \cite[27.16]{GMS} and \cite[17.1]{GMS}). Clearly any invariant of $I$ restricts to
an invariant of $\Quad_m$ for any even $m$, and we want to express it in
terms of the given basis.
\\

For practical purposes it is more convenient to introduce a different
basis for $\Inv(\Quad_m, W)$ which is the equivalent of the Stiefel-Whitney
classes for Witt invariants. We use the notations and definitions from
section 1. Recall from \cite[§5]{EKM} that the total Stiefel-Whitney
map $w_t: GW(K)\to \Lambda(H^*(K,\mu_2))$ is the only group morphism
such that $w_t(\fdiag{a})=1+(a)t$ for all $a\in K^*$. We generalize
this construction:

\begin{propdef}
  There is a unique group morphism
  \[ \foncdef{h_t}{GW(K)}{\Lambda(A(K))}{x}{h_t(x) = \sum_{d\in \N} h^d(x)t^d} \]
  such that $h_t(\fdiag{a}) = 1 + \{a\}t$ for all $a\in K^*$. The map $h^d$ takes
  values in $A^{\pgq d}(K)$. For any $m\in \N^*$, we write $h_m^d: \Quad_m(K)\to A(K)$
  for the restriction of $h^d$ to forms of dimension $m$.
  Then $h_m^d\in \Inv(\Quad_m,A^{\pgq d})$.

  If $A=H$, then $h^d$ is the Stiefel-Whitney map $w_d$. If $A=W$, we write
  $P^d=h^d$ and  $P_m^d = h_m^d$; then for any $q\in \Quad_m(K)$:
  \begin{equation}\label{eq_p}
    P^d_m(q) = \sum_{k=0}^d (-1)^k \binom{m-k}{d-k} \lambda^k(q). 
  \end{equation}

  In both cases, $(h_m^d)_{0\ppq d\ppq m}$ is a basis of the $A(k)$-module
  $\Inv(\Quad_m,A)$.
\end{propdef}

\begin{proof}
  The uniqueness of $h_t$ is obvious since $GW(K)$ is generated by the $\fdiag{a}$
  as an additive group. For $A=H$, the existence can either be deduced from
  the case $A=W$, or from the classical existence of Stiefel-Whitney maps.
  For $A=W$, we define $P^d$ piecewise on quadratic forms, using formula
  (\ref{eq_p}) for $P^d_m$ in each dimension $m$. We see immediately
  from the definition that $P_1^d(\fdiag{a})$ is 1 if $d=0$, $\pfis{a}$ if $d=1$,
  and $0$ if $d\pgq 2$. The fact that this extends to a group morphism
  $GW(K)\to \Lambda(GW(K))$ can be deduced using
  the universal property of Grothendieck groups if we can show that for any
  $q\in \Quad_m(K)$, $q'\in \Quad_n(K)$, we have:
  \[P_{m+n}^d(q+q')= \sum_{k=0}^d P_m(q)P_n(q'). \]
  And indeed we find:
  \begin{align*}
    \sum_{k=0}^d P_m(q)P_n(q') &= \sum_{k=0}^d\sum_{i=0}^k\sum_{j=0}^{d-k}
       (-1)^{i+j}\binom{m-i}{k-i}\binom{n-j}{d-k-j}\lambda^i(q)\lambda^j(q') \\
    &= \sum_{l=0}^d(-1)^l\sum_{i+j=l}\left( \sum_{k=i}^{d-j}\binom{m-i}{k-i}\binom{n-j}{d-k-j}\right) \lambda^i(q)\lambda^j(q') \\
    &= \sum_{l=0}^d (-1)^l\binom{m+n-l}{d-l} \sum_{i+j=l}\lambda^i(q)\lambda^j(q') \\
    &= P_{m+n}(q+q').
  \end{align*}
  From the group property we easily see that
  \begin{equation}\label{eq_p_pfis}
    h_m^d(\fdiag{a_1,\dots,a_m}) = \sum_{i_1<\dots <i_d} \{a_{i_1},\dots,a_{i_d}\}
  \end{equation}
  so $h_m^d(q)\in I^d(K)$ if $q\in \Quad_m(K)$. The fact that $h_m^d$ is an
  invariant is obvious given the definition with the $\lambda$-powers, or
  can be deduced from the uniqueness statement. Finally, the fact that the
  $h_m^d$ for a basis of $\Inv(\Quad_m,A)$ is a consequence of Serre's result,
  directly for $A=H$, and observing for $A=W$ that the transition matrix from
  $(P_m^d)_{0\ppq d\ppq m}$ to $(\lambda^d)_{0\ppq d\ppq m}$ is triangular unipotent.
\end{proof}

\begin{rem}
  Note that this does not define a pre-$\lambda$-ring structure on $GW(K)$
  since $P^1$ is not the identity (indeed, $P^1(\fdiag{a})=\pfis{a}$).
\end{rem}

Then we can state:

\begin{prop}\label{prop_fixed_dim}
  Let $m=2r\in \N^*$, $d\in \N$ and $q\in \Quad_m(K)$. Then :
  \begin{align*}
    f_1^d(q) &= \sum_{i=0}^d (-1)^i \binom{r-i}{d-i}\{-1\}^{d-i} h_m^i(q) \\
    g_1^d(q) &= \sum_{i=0}^d (-1)^i\binom{r-i-1 + \lfloor \frac{d+1}{2} \rfloor}{d-i} \{-1\}^{d-i} h_m^i(q). 
  \end{align*}
\end{prop}

\begin{proof}
  We prove the statement concerning $f_1^d$; the case of $g_1^d$
  may be deduced by a lengthy but straightforward computation
  using proposition \ref{prop_f_g}, or can be directly proved by
  the same method.

  Write $\alpha_m^d$ for the invariant of $\Quad_m$ defined by the
  right-hand side of the equation. It is clear by definition that $\alpha_m^0 = 1$
  coincides with $f_1^0$ on $\Quad_m$. We claim that it is enough to
  show that for any $d\in \N^*$:
  \begin{equation}\label{eq_alpha_norm}
    \alpha_2^d(\pfis{1})=0,
  \end{equation}
  and for any $m=2r\in \N^*$, any $q\in \Quad_m(K)$ and any $a\in K^*$:
  \begin{equation}\label{eq_alpha_add}
    \alpha_{m+2}^d(q+\pfis{a}) = \alpha_m^d(q)+\{a\} \alpha_m^{d-1}(q).
  \end{equation}
  Indeed, taking $a=1$ in (\ref{eq_alpha_add}) shows that $\alpha_m^d(q)$
  depends only on the Witt class of $q\in \Quad_m(K)$, so it defines an invariant
  $\alpha^d\in M(1)$. Then (\ref{eq_alpha_norm}) shows that $\alpha^d$ is
  normalized, and (\ref{eq_alpha_add}) shows that $(\alpha^d)^+=\alpha^{d-1}$,
  so by an immediate induction $\alpha^d = f_1^d$.

  From the formula (\ref{eq_p_pfis}) we easily see that $h_2^0(\pfis{a})=1$, 
  $h_2^1(\pfis{a})=\{-a\}$, and $h_2^i(\pfis{a})=0$ if $i\pgq 2$. Thus
  $\alpha_2^d(\pfis{1}) = \{-1\}^d - \{-1\}^{d-1}\cdot \{-1\}=0$
  which shows (\ref{eq_alpha_norm}).

  Furthermore, if $i\in \N$ and $q\in \Quad_m(K)$:
  \begin{align}
    h_{m+2}^i(q+\pfis{a}) &= h_m^i(q)+\{-a\} h_m^{i-1}(q) \nonumber \\
                    &= \left(h_m^i(q)+\{-1\}h_m^{i-1}(q)\right) -\{a\}h_m^{i-1}(q)
  \end{align}
  (where by convention $h_m^{-1}=0$), therefore:
  \begin{align*}
    \alpha_{m+2}(q+\pfis{a})
        =& \sum_{i=0}^d (-1)^i\binom{r+1-i}{d-i}\{-1\}^{d-i}\left(h_m^i(q)+\{-1\}h_m^{i-1}(q)\right) \\
         &- \{a\} \sum_{i=0}^d (-1)^i\binom{r+1-i}{d-i}\{-1\}^{d-i}h_m^{i-1}(q) \\
        =& \sum_{i=0}^{d-1}\left( (-1)^i\binom{r-i+1}{d-i}+(-1)^{i+1}\binom{r-i}{d-i-1}\right)\{-1\}^{d-i} h_m^i(q)\\
         &+ (-1)^d h_m^d(q) -\{a\} \sum_{i=0}^{d-1}(-1)^{i+1}\binom{r-i}{d-i-1} \{-1\}^{d-i-1}h_m^i(q) \\
        =& \sum_{i=0}^d (-1)^i\binom{r-i}{d-i}\{-1\}^{d-i}h_m^i(q) \\
         &+ \{a\} \sum_{i=0}^{d-1}(-1)^i\binom{r-i}{d-1-i} \{-1\}^{d-1-i}h_m^i(q)
  \end{align*}
  which gives the expected formula.
\end{proof}

\begin{rem}\label{rem_fixed_dim}
  In particular, looking carefully at the binomial coefficients in the formula
  and remembering that $h_m^i=0$ if $i>m$, we retrieve the fact that
  $g_1^d$ is zero if $d>m$ (recall corollary \ref{cor_g_fixed_dim}). On the other
  hand, we see that $f_1^d$ can be non-zero for arbitrarily high values of $d$,
  even for fixed $m$.
\end{rem}

\begin{rem}\label{rem_sw_moins_1}
  If $-1$ is a square in $k$, then $f_1^d = g_1^d = h_m^d$ on
  $\Quad_m$ for any even $m\in \N^*$.
\end{rem}

\begin{coro}\label{cor_fix_dim}
  For any even $m\in \N^*$, the restrictions of $f_1^d$ (or $g_1^d$)
  for $0\ppq d\ppq m$ form an $A(k)$-basis of $\Inv(\Quad_m, A)$.
  In particular, any invariant of $\Quad_m$ can be extended to $I$.
\end{coro}

\begin{rem}
  Serre also describes the cohomological invariants of $\Quad_{m,\delta}$,
  meaning of forms with prescribed determinant $\delta$, and in particular
  this gives a description of invariants of $\Quad_m\cap I^2$.
  They are given by Stiefel-Whitney classes, plus one invariant that
  does not extend to $\Quad_m$ in general. Since any invariant
  of $I^2$ extends to $I$, this shows that there are invariants of
  $\Quad_m\cap I^2$ that do not extend to $I^2$.

  There are also examples in the literature of some invariants
  of $\Quad_m\cap I^3$ that one can show using the results
  in this article do not extend to $I^3$ (for instance
  the invariant $a_5$ mentioned in section \ref{sec_factor}).
\end{rem}

\begin{rem}
  Let us consider the cohomological invariants of $\Quad_m/\sim$
  (the similarity classes of quadratic forms of dimension $n$).
  This is of course the same thing as an invariant of $\Quad_m$
  which is constant on similarity classes, so according to corollary
  \ref{cor_fix_dim} any such invariant is a unique combination of the $v_d^{(1)}$
  with $0\ppq d\ppq m$. Now corollary \ref{cor_inv_sim} shows
  that such a combination is constant on similarity classes iff
  the only $d$ that appear are odd. This is exactly the description
  that Rost gives in \cite[lem 2]{R98}, where he proves that any invariant
  of $\Quad_m/\sim$ is a unique combination of invariants he calls
  $v_{2i+1}$, and a simple computation shows that $v_{2i+1}=v_{2i+1}^{(1)}$.

  On the other hand, our present tools cannot a priori describe all
  cohomological invariants of similarity classes in $\Quad_m\cap I^2$,
  since not all invariants of $\Quad_m\cap I^2$ extend to $I^2$.
  What we can say from the previous remark and corollary \ref{cor_inv_sim}
  is that those which do extend to $I^2$ can be uniquely written
  as $\sum_{d=0}^r a_d\cup v_{2d}^{(2)}$ with $(-1)\cup a_d=0$
  if $d>0$ is even. However, Rost describes in \cite[thm 6]{R98} the
  invariants of similarity classes $\Quad_m\cap I^2$, and proves
  that they are combinations of invariants $\eta_d$. In turns out
  that $\eta_d = v_{2d}^{(2)}$, so this shows that even though some invariants of
  isometry classes in $\Quad_m\cap I^2$ do not extend to $I^2$, all
  invariants of \emph{similarity classes} in $\Quad_m\cap I^2$ do
  extend to $I^2$ (and therefore to $I$), and Rost's description is
  exactly the same as ours.
\end{rem}

\section{Operations on mod 2 cohomology}\label{sec_operations}

In this section we are specifically interested in cohomological
invariants. It was observed by Serre that one may define some
sorts of divided squares on mod 2 cohomology : 

\[ \anonfoncdef{H^n(K,\mu_2)}{H^{2n}(K,\mu_2)/(-1)^{n-1}\cup H^{n+1}(K,\mu_2)}
{\sum_i \alpha_i}{\sum_{i<j} \alpha_i\cup \alpha_j.} \]

The quotient on the right-hand side is necessary for the map
to be well-defined. Similarly, one may define higher divided
powers :

\[ \anonfoncdef{H^n(K,\mu_2)}{H^{dn}(K,\mu_2)/(-1)^{n-1}\cup H^{(d-1)n+1}(K,\mu_2)}
{\sum_i \alpha_i}{\sum_{i_1<\dots <i_d} \alpha_{i_1}\cup\cdots \cup \alpha_{i_d}.} \]

On the other hand, Vial (\cite{Via}) characterizes natural
operations
\[ H^n(K,\mu_2) \To H^*(K,\mu_2) \]
(his statement is formulated for mod 2 Milnor K-theory, which is
equivalent according to the resolution of Milnor's conjecture).
The precise statement, slightly reformulated, is the following (the
original statement forgets to explicitly assume that operations must
have uniformly bounded degree):

\begin{prop}[\cite{Via}, Theorem 2]
  If $n\in \N^*$, the $H^*(k,\mu_2)$-module of operations
  $H^n(K,\mu_2)\to H^*(K,\mu_2)$ with uniformly bounded degree is
  \[ H^*(k,\mu_2)\cdot 1 \oplus H^*(k,\mu_2)\cdot \Id \oplus  \bigoplus_{d\in \N} \Ker(\tau_n)\cdot \theta_d \]
  where $\tau_n: H^*(k,\mu_2)\To H^*(k,\mu_2)$ is defined by $\tau_n(x) = (-1)^{n-1}\cup x$
  and if $a\in \Ker(\tau_n)$, then
  \[ a\cdot \theta_d \left(\sum_{1\ppq i\ppq r} x_i\right) = a\cdot \sum_{i_1<\cdots <i_d}x_{i_1}\cup \cdots \cup x_{i_d} \]
  where the $x_i$ are symbols.
\end{prop}

Note that the ``divided power operation'' $\theta_d$ is not defined on its own, but
$a\cdot \theta_d$ is well-defined when $a\in \Ker(\tau_n)$. This is
similar to how for Serre's operations it was necessary to consider some
quotient on the right-hand side of the map; here one has to put some restriction
on the left-hand side, in both cases to annihilate appropriate
powers of the symbol $(-1)\in H^1(K,\mu_2)$.
The remarkable phenomenon is that when we work on the level of $I^n$,
we can lift those $\theta_d$ with no restriction: this is our
$u_{nd}^{(n)}$.

Moreover, it is not too difficult to retrieve Vial's theorem using
our results about invariants of $I^n$: operations on $H^n(K,\mu_2)$
are none other than invariants $\alpha\in M(n)$ (with $A=H$) such that
\begin{equation}\label{eq_vial}
\alpha(q+\phi)=\alpha(q) \quad \forall q\in I^n(K), \phi\in \Pf_{n+1}(K).
\end{equation}

Consider the following easy lemma :

\begin{lem}
  Let $n\in \N^*$, and let us restrict to $A=H$. For any $\alpha\in M(n)$,
  any $q\in I^n(K)$ and any $\phi\in \Pf_{n+1}(K)$, we have
  \[ \alpha(q+\phi) = \alpha(q) + (-1)^{n-1}\cup e_{n+1}(\phi)\cup \alpha^{++}(q). \]
\end{lem}

\begin{proof}
  Up to taking linear combinations, we may restrict to the
  case of $\alpha = u_{nd}^{(n)}$. Using corollary \ref{cor_restr_cohom},
  we see that $u_{nd}^{(n)}(\phi)$ is $1$ if $d=0$, $(-1)^{n-1}\cup e_{n+1}(\phi)$
  if $d=2$, and $0$ otherwise. Then using the sum formula for $u_{nd}^{(n)}$
  we find
  \[ u_{nd}^{(n)}(q+\phi)= u_{nd}^{(n)}(q) + (-1)^{n-1}\cup e_{n+1}(\phi)\cup u_{n(d-2)}^{(n)}(q). \qedhere \]
\end{proof}

Then $\alpha\in M(n)$ satisfies condition (\ref{eq_vial}) if and only if
$(-1)^{n-1}\cup \alpha^{++}=0$, which precisely means that if we write
$\alpha = \sum_d a_d\cup u_{nd}^{(n)}$ then, for $d\pgq 2$, $a_d\in \Ker(\tau_n)$,
and we indeed retrieve Vial's description.

\section{Invariants of semi-factorized forms}\label{sec_factor}

In \cite[§20]{Gar}, Garibaldi defines a cohomological invariant on $\Quad_{12}\cap I^3$
the following way: any such form can be written $q=\pfis{c}q'$ where $q'\in I^2(K)$,
and we set $a_5(q)=e_5(\pfis{c} \bar{\pi}_2^2(q')) = (c)\cup u_4^{(2)}(q')$
(using our notation). Of course, the non-trivial ingredient is that
$\pfis{c} \bar{\pi}_2^2(q')$ is actually independent of the decomposition of $q$.

This construction does not correspond to any of the tools we developped
so far, since it does not give an invariant of $I^3$. However, it
is easy to see that the construction works for any Witt class $q\in I^3(K)$
that factorizes as $q=\pfis{c}q'$. This leads us to the more general
definition:

\begin{defi}
  Let $n\in \N^*$ and $r\in \N$ such that $r\ppq n$. We set
  \[ I^{n,r}(K) = \ens{\phi\cdot q}{\phi\in \Pf_r(K),\, q\in I^{n-r}(K)}. \]
  We also define $M(n,r) = \Inv(I^{n,r},A)$, and similarly $M_0(n,r)$,
  $M^{\pgq d}(n,r)$ and $M_0^{\pgq d}(n,r)$.
  In particular, $I^{n,0} = I^n$, so $M(n,0)=M(n)$ and so on.
\end{defi}

\begin{rem}\label{rem_milnor}
  A consequence of Milnor's conjecture proved in \cite[41.7]{EKM} is that
  $I^{n,r}(K)= I^{r,r}(K)\cap I^n(K)$, so in particular
  $I^{n,r}(K)\cap I^{n+1}(K)=I^{n+1,r}(K)$. 
\end{rem}

Clearly, if $(m,s)\pgq (n,r)$, then $I^{m,s}(K)\subset I^{n,r}(K)$, so
we have a restriction morphism
\[ \foncdef{\rho_{(n,r),(m,s)}}{M(n,r)}{M(m,s)}{\alpha}{\alpha_{|I^{m,s}}} \]
which is a morphism of filtered $A(k)$-algebras, and sends
$M_0(n,r)$ to $M_0(m,s)$. In particular, when $r=s=0$, we retrieve
the restriction morphism $\rho_{n,m}$ defined in (\ref{eq_rho}).
We usually drop the indexes and simply write $\rho: M(n,r)\to M(m,s)$,
since the indexes can be infered from the source and target modules.

We can also define a morphism that goes in the other direction:
  
\begin{propdef}
  Let $n,r,t\in \N$ with $t\ppq r<n$.
  There is a unique morphism of filtered $A(k)$-modules
  \[  \foncdef{\Delta_{(n,r)}^t}{M(n,r)}{M(n \minus t,r\minus t)[-t]}{\alpha}{\alpha^{(t)},} \]
  such that $\alpha^{(t)}(0)=\alpha(0)$, and if $\alpha\in M_0(n,r)$ then
  \[ \alpha(\phi\cdot q) = f_t(\phi)\cdot \alpha^{(t)}(q) \]
  for any $\phi\in \Pf_t(K)$, and $q\in I^{n-t,r-t}(K)$.
  Furthermore, $\Delta_{(n,r)}^t$ is injective.
\end{propdef}

\begin{proof}
  Since $M(n,r) = A(k)\oplus M_0(n,r)$, this piecewise definition of
  $\Delta_{(n,r)}^t$ determines the whole function. Let $\alpha\in M^{\pgq d}_0(n,r)$,
  and $q\in I^{n-t,r-t}(K)$. Then $\phi\mapsto \alpha(\phi\cdot q)$
  defines an invariant of $\Pf_t$ over $K$ with values in $A^{\pgq d}$.
  Using lemma \ref{lem_a_inv}, there are unique $x(q),y(q)\in A(K)$ such that
  \[  \alpha(\phi\cdot q) = x(q) + f_t(\phi)\cdot y(q)  \]
  and by uniqueness those are invariants of $I^{n-t,r-t}$, with $x = \alpha(0)=0$.
  We then set $\alpha^{(t)} := y$. Furthermore, using lemma \ref{lem_a_filtr},
  we see that $y(q)\in A^{\pgq d-t}(K)$, so $\alpha^{(t)}\in M_0^{\pgq d-t}(n\minus t,r\minus t)$.
  The injectivity is clear since any element of $I^{n,r}(K)$ is of the form
  $\phi q$ with $\phi$ and $q$ as in the statement, and $\alpha(\phi q)$ is
  determined by $\alpha^{(t)}$.
\end{proof}

We usually drop the indexes and simply write $\Delta^t: M(n,r)\to M(n \minus t,r\minus t)[-t]$.
Using this notation, it is clear by definition that $\Delta^t\circ \Delta^{t'}=\Delta^{t+t'}$.
The natural question is then:
\\

\textbf{Question:} What is the image of $\Delta^t:M(n+t,r+t)\to M(n,r)$ ?
\\

This can be rephrased as: for which $\beta\in M_0(n,r)$ is it true that for
all $\phi\in \Pf_t(K)$ and $q\in I^{n,r}(K)$, $f_t(\phi)\beta(q)$
only depends on $\phi q$? With this point of view, the existence of the invariant $a_5$
given at the beginning of the section (which is \cite[20.7]{Gar}) is exactly equivalent to the
fact that $f_2^2\in M(2,0)$ is in the image of of $\Delta^1: M(3,1)\to M(2,0)$.
The main result of the section is a generalization of this fact:

\begin{thm}\label{thm_delta}
  For any $n\in \N^*$, $\Delta^1: M_0(n+1,1)\to M_0(n)[-1]$ is an isomorphism
  of filtered $A(k)$-modules.
\end{thm}

\begin{rem}
  This means that $\Delta^1: M_0(n+1,1)\to M_0(n)[-1]$ is a module isomorphism,
  but it is not a \emph{filtered} module isomorphism, since it is the identity
  on the constant components, and while the identity is a bijective filtered
  morphism from $A(k)$ to $A(k)[-1]$, it is of course not a filtered isomorphism.
\end{rem}

Before we prove theorem \ref{thm_delta}, we construct a common generalization
of $\rho$ and $\Delta^t$, which allows to make simple statements about the
general properties of both those morphisms. Most of that is not useful for the
proof of the theorem, but has some independant interest. 

\begin{defi}
  Let $m,n\in \N^*$ and $r,s\in \N$ such that $r<n$ and $s<m$. We say that
  a filtered $A(k)$-module morphism $M(n,r)\to M(m,s)[-t]$ is of type $\Omega^t$
  if it is a composition of morphisms $\omega_i: M(n_i,r_i)[-a_i]\to M(n_{i+1},r_{i+1})[-a_i-t_i]$
  for $i=0,\dots,d$, with $(n_0,r_0)=(n,r)$, $a_0=0$, $(n_{d+1},r_{d+1})=(m,s)$,
  $t=\sum_i t_i$, and $\omega_i$ is either $\rho$ (so $t_i=0$) or $\Delta^{t_i}$.
  
  In particular, we define $\omega$ of type $\Omega^1$:
  \[ \omega: M(n,r) \xrightarrow{\rho} M(n+1,r+1) \xrightarrow{\Delta^1} M(n,r)[-1]. \]
\end{defi}

\begin{rem}\label{rem_omega}
  It is not difficult to see that there exists a morphism $M(n,r)\to M(m,s)[-t]$
  of type $\Omega^t$ iff $t\pgq n-m$ and $t\pgq r-s$. 
\end{rem}

\begin{prop}\label{prop_omega}
  Let $m,n\in \N^*$ and $r,s\in \N$ such that $r<n$ and $s<m$, and let
  $t\in \N$ be such that $t\pgq n-m$ and $t\pgq r-s$. Then there is
  exactly one morphism $M(n,r)\to M(m,s)[-t]$ of type $\Omega^t$, and
  we call it simply $\Omega^t$. The morphism $\Omega^t: M(n,r)\to M(n,r)[-t]$
  is $\omega^t$.

  In particular, let $t'\pgq t$. Then the following diagram of filtered $A(k)$-modules
  commutes:
  \[ \begin{tikzcd}
      & M(n,r) \dlar[swap]{\Omega^t} \drar{\Omega^{t'}} & \\
    M(m,s)[-t] \arrow{rr}{\omega^{t'-t}} & & M(m,s)[-t'].
  \end{tikzcd} \]
\end{prop}

\begin{proof}
  The only thing to prove is that there is at most one morphism of type
  $\Omega^t$. The fact that $\Omega^t=\theta^t$ then follows since $\omega^t$
  is of type $\Omega^t$ by definition, and the commutativity of the diagram
  comes from the fact that both compositions are of type $\Omega^{t'}$.

  To show this uniqueness, it is enough to show that the following diagram
  commutes whenever it makes sense:
  \[ \begin{tikzcd}
    M(n,r) \rar{\rho} \dar[swap]{\Delta^t} & M(m,s) \dar{\Delta^t} \\
    M(n\minus t,r\minus t)[-t] \rar{\rho} & M(m\minus t, s\minus t)[-t].
  \end{tikzcd} \]
  Indeed, if we can prove this, then we can show by induction on the
  length of the composition that in the definition of a morphism of type
  $\Omega^t$ we can always assume that the first morphisms are all of the
  form $\rho$, and the remaining ones are all of the form $\Delta^{t_i}$.
  But then the result is clear, since a composition of restriction morphisms
  is a restriction morphism, and $\Delta^t\circ \Delta^s=\Delta^{t+s}$ (with
  the only indices that make sense), so the morphism is entirely characterized
  by its source, its target, and $t$.

  We now show that the diagram commutes. Let $\alpha\in M(n,r)$. Since all
  morphisms are the identity on the constant components, we may assume that
  $\alpha\in M_0(n,r)$. Let us write $\beta = (\alpha^{(t)})_{|I^{m\minus t,s\minus t}}$,
  and let $\phi\in \Pf_t(K)$, $\psi\in \Pf_{s-t}(K)$ and $q\in I^{m-s}(K)$.
  We can set $\psi=\psi_1\psi_2$ with $\psi_1\in \Pf_{r-t}(K)$ and $\psi_2\in \Pf_{s-r}(K)$;
  then if $q'=\psi_2q\in I^{m-r}(K)$, we have
  \[ \alpha(\phi\psi q) = \alpha(\phi\psi_1 q') = f_t(\phi)\alpha^{(t)}(\psi_1q') = f_t(\phi)\beta(\psi q), \]
  which shows that $\beta = (\alpha_{|I^{m,s}})^{(t)}$.
\end{proof}

\begin{ex}
  The morphism $\Omega^0: M(n,r)\to M(m,s)$ exists when $(m,s)\pgq (n,r)$, and
  of course it is the restriction morphism $\rho$. The morphism
  $\Omega^t: M(n,r)\to M(n\minus t,r\minus t)[-t]$ exists when $t\ppq r$,
  and it is $\Delta^t$.
\end{ex}

\begin{ex}
  There is a morphism $\Omega^t: M(n)\to M(m)[-t]$ when $t\pgq n-m$,
  and if $n=m$ it is $\omega^t$, with $\omega: M(n)\to M(n)[-1]$.
\end{ex}

We can now collect some basic properties of the morphisms $\Omega^t$.

\begin{prop}
  Let $n,m,r,s,t\in \N$ be as is proposition \ref{prop_omega}.
  Then for any $\alpha,\beta\in M_0(n,r)$, we have
  \[ \Omega^t(\alpha\beta) = \{-1\}^t \Omega^t(\alpha)\Omega^t(\beta). \]
\end{prop}

\begin{proof}
  Since the restriction morphisms obviously preserve the product of invariants,
  we may assume that $\Omega^t=\Delta^t$. Then for any $\phi\in \Pf_t(K)$,
  $\psi\in \Pf_{r-t}(K)$ and $q\in I^{n-r}(K)$, we have
  \begin{align*}
    (\alpha\beta)(\phi\psi q) &= (f_t(\phi)\alpha^{(t)}(\psi q))(f_t(\phi)\beta^{(t)}(\psi q)) \\
    &= \{-1\}^tf_t(\phi) (\alpha^{(t)}\beta^{(t)})(\psi q),
  \end{align*}
  hence the result.
\end{proof}

We may note from proposition-definition \ref{def_psi} that we have well-defined
filtered morphisms
\[ \Psi: M(n,r)\To M(n,r)[-1] \]
for any $n,r\in \N$ such that $r<n$.

\begin{prop}
  Let $n,m,r,s,t\in \N$ be as is proposition \ref{prop_omega}. Then the following
  diagram of filtered $A(k)$-modules commutes:
  \[ \begin{tikzcd}
    M(n,r) \rar{\Omega^t} \dar[swap]{\Psi} & M(m,s)[-t] \dar{\Psi} \\
    M(n,r)[-1] \rar{\Omega^t} & M(m,s)[-t-1].
    \end{tikzcd} \]
\end{prop}

\begin{proof}
  The definition of $\Psi$ makes it clear that it commutes with restriction
  morphisms, since it is defined on the whole $\Inv(W,A)$. Thus we
  may assume that $\Omega^t=\Delta^t$. Let $\alpha\in M(n,r)$, $\phi\in \Pf_t(K)$,
  $\psi\in \Pf_{r-t}(K)$, $q\in I^{n-r}(K)$, and $\lambda\in K^*$. Then:
  \begin{align*}
    \alpha(\fdiag{\lambda}\phi\psi q) &= f_t(\phi) \alpha^{(t)}(\fdiag{\lambda}\psi q) \\
                   &= f_t(\phi) \alpha^{(t)}(\psi q) + f_t(\phi)\{\lambda\} \tld{\alpha^{(t)}}(\psi q)
  \end{align*}
  but also
  \begin{align*}
    \alpha(\fdiag{\lambda}\phi\psi q) &= \alpha(\phi\psi q) + \{\lambda\}\tld{\alpha}(\phi\psi q) \\
                     &= f_t(\phi) \alpha^{(t)}(\psi q) + \{\lambda\}f_t(\phi) \tld{\alpha}^{(t)}(\psi q)
  \end{align*}
  which gives $\tld{\alpha^{(t)}} = \tld{\alpha}^{(t)}$.
\end{proof}

Since we saw in corollary \ref{cor_restr_w} that $\Phi^+$ is far from commuting
with the restriction morphisms, we cannot expect such a good compatibility with
the morphisms $\Omega^t$, but we still get:

\begin{prop}\label{prop_omega_phi}
  Let $n\in \N^*$ and let $t\in \N$ be such that $t<n$. Then the following
  diagram of filtered $A(k)$-modules commutes for any $\eps=\pm 1$:
  \[ \begin{tikzcd}
    M(n) \rar{\Omega^t} \dar[swap]{\Phi^\eps} & M(n\minus t)[-t] \dar{\Phi^\eps} \\
    M(n)[-n] \rar{\{-1\}^t \Omega^t} & M(n\minus t)[-n].
    \end{tikzcd} \]
\end{prop}

\begin{proof}
  The diagram obviously commutes for the constant components (since we
  find 0 in both cases), so we may consider $\alpha\in M_0(n)$. Let
  $\phi\in \Pf_t(K)$, $\psi\in \Pf_{n-t}(K)$ and $q\in I^{n-t}(K)$.
  Then:
  \begin{align*}
    \alpha(\phi(q+\eps\psi)) &= \alpha(\phi q) + \eps f_n(\phi\psi) \alpha^\eps(\phi q) \\
                         &= f_t(\phi)\alpha^{(t)}(q) + \eps \{-1\}^tf_n(\phi\psi)(\alpha^+)^{(t)}(q)
  \end{align*}
  as well as
  \begin{align*}
    \alpha(\phi(q+\eps\psi)) &= f_t(\phi)\alpha^{(t)}(q+\eps\psi) \\
                         &= f_r(\phi)\alpha^{(r)}(q) + \eps f_r(\phi)f_{n-r}(\psi)(\alpha^{(r)})^+(q)
  \end{align*}
  which proves that $((\alpha_{|I^{n,t}})^{(t)})^+ = \{-1\}^t(\alpha^+)_{|I^{n,t}}^{(t)}$.
\end{proof}

\begin{coro}\label{cor_divisible}
  Let $n,t\in \N$ such that $t< n$. Then for any $d\in \N^*$,
  the morphism $\Omega^t: M(n)\to M(n\minus t)[-t]$ satisfies
  \[  \Omega^t(f_n^d) = \{-1\}^{t(d-1)}f_{n-t}^d.  \]
  In particular, if $\phi\in \Pf_t(K)$ and $q\in I^n(K)$ is a multiple
  of $\phi$, then $f_n^d(q)$ is a multiple of $f_t(\phi)$.
\end{coro}

\begin{proof}
  The formula follows from an induction of $d$, using proposition
  \ref{prop_omega_phi}. For the last statement, note that according
  to remark \ref{rem_milnor}, there is $q'\in I^{n-t}(K)$ such that
  $q=\phi q'$. Then according to the formula,
  \[ f_n^d(q) = f_n^d(\phi q') = \{-1\}^{t(d-1)}f_t(\phi)f_{n-t}^d(q'). \qedhere \]
\end{proof}

We now turn to the proof of theorem \ref{thm_delta}. We first need a
preliminary lemma:

\begin{lem}\label{lem_thm_delta}
  Let $a,b\in K^*$, and $q\in \hat{I}(K)$ of the form
  \[ q = \sum_{i=1}^r \fdiag{x_i}\gpfis{c_i} \]
  where $c_i$ is represented by $\pfis{ab}$. Then for any $k\in \N^*$, 
  \[ \pfis{a}\lambda^k(q) = \pfis{b}\lambda^k(q). \]
  In particular, for any $n,d\in \N^*$, $\pfis{a}\pi_n^d(q) = \pfis{b}\pi_n^d(q)$.
\end{lem}

\begin{proof}
  We have
  \[ \lambda^k(q) = \sum_{d_1+\cdots+d_r=k}
    \lambda^{d_1}(\fdiag{x_1}\gpfis{c_1})\cdots \lambda^{d_r}(\fdiag{x_r}\gpfis{c_r}). \]
  Now at least one of the $d_i$ is non-zero, so we may conclude since
  \begin{align*}
    \pfis{a} \lambda^{d_i}(\fdiag{x_i}\gpfis{c_i}) &= \fdiag{x_i^d}\pfis{a}\gpfis{c_i} \\
         &= \fdiag{x_i^d}\pfis{b}\gpfis{c_i} \\
         &= \pfis{b} \lambda^{d_i}(\fdiag{x_i}\gpfis{c_i}),
  \end{align*}
  where we use lemma \ref{lem_lambda_gpfis}, and the fact that if $c$ is represented
  by $\pfis{ab}$ then $\pfis{a,c}=\pfis{b,c}$. The statement about $\pi_n^d$
  follows since by definition $\pi_n^d$ is a combination of the $\lambda^k$ with
  $1\ppq k\ppq d$.
\end{proof}

We can finally prove:

\begin{proof}[Proof of theorem \ref{thm_delta}]
  It suffices to show that $f_n^d$ is in the image of $\Delta^1$
  for all $d\pgq 1$, which amounts to say that $\pfis{a}q\mapsto \{a\}f_n^d(q)$
  is well-defined, in other words that if $q,q'\in I^n(K)$ and $a,b\in K^*$,
  then $\pfis{a}q = \pfis{b}q'$ implies $\{a\}f_n^d(q) = \{b\} f_n^d(q')$.

  Assume first that $a=b$. Then according to \cite[6.23]{EKM}, 
  \[ q - q' = \sum_{i\in J} \pfis{c_i}q_i \]
  where $q_i\in W(K)$ and $c_i$ is represented by $\pfis{a}$. We may then reason by induction on $|J|$,
  and we are reduced to the case where $q' = q +\pfis{c}q_0$, with $c$ represented
  by $\pfis{a}$. But according to corollary \ref{cor_divisible}, for any $k\in \N^*$,
  $f_n^k(\pfis{c}q_0)$ is divisible by $\{c\}$, so $\{a\}f_n^k(\pfis{c}q_0) = 0$.
  From there:
  \begin{align*}
    \{a\}f_n^d(q') &= \{a\}\sum_{k=0}^d f_n^k(q)f_n^{d-k}(\pfis{c}q_0)\\
                       &= \{a\}f_n^d(q).
  \end{align*}

  Suppose now $a\neq b$. Then Hoffmann shows in \cite[B.5]{Gar} that we have
  \[ \pfis{a}q = \pfis{a}q_0 = \pfis{b}q_0 = \pfis{b}q' \]
  where $q_0 = \sum_{i\in J} \fdiag{x_i}\pfis{c_i}\in I^n(K)$, and $c_i$ is represented
  by $\pfis{ab}$. The previous discussion shows that $\{a\}f_n^d(q) = \{a\}f_n^d(q_0)$
  and $\{b\}f_n^d(q) = \{b\}f_n^d(q_0)$, so it just remains to show 
  that $\{a\}f_n^d(q_0) = \{b\}f_n^d(q_0)$ for any $q_0$
  admitting a decomposition as above, which is a direct consequence of
  lemma \ref{lem_thm_delta}.
\end{proof}

\bibliographystyle{plain}
\bibliography{invariants_witt}

\end{document}